\documentclass[11pt]{amsart}

\newcommand{\Addresses}{{
  \bigskip
  \footnotesize  
  
  \noindent Gabriele Viaggi, \textsc{Mathematical Institute, University of Pisa, Pisa}\par\nopagebreak
  \textit{E-mail address}: \texttt{gabriele.viaggi@unipi.it}
  }}

\usepackage[margin=1.5in]{geometry}
\usepackage[colorlinks,citecolor=cyan,linkcolor=blue]{hyperref}

\usepackage{dsfont}
\usepackage{amsmath}
\usepackage{amsthm}
\usepackage{amssymb}
\usepackage[all,cmtip]{xy}
\usepackage{enumerate}
\usepackage{overpic}
\usepackage{float}
\usepackage{xcolor}

\newcommand{\ep}{
\epsilon
}
\newcommand{\mc}[1]{
\mathcal{#1}
}
\newcommand{\mb}[1]{
\mathbb{#1}
}
\newcommand{\T}{
\mc{T}
}

\newtheorem*{thm*}{Theorem}
\newtheorem*{cor*}{Corollary}
\newtheorem*{pro*}{Proposition}

\newtheorem{thmA}{Theorem}

\newtheorem{thm}{Theorem}[section]

\newtheorem{cor}[thm]{Corollary}

\newtheorem{lem}[thm]{Lemma}
\newtheorem{pro}[thm]{Proposition}

\theoremstyle{definition}

\newtheorem*{defi*}{Definition}
\newtheorem{dfn}[thm]{Definition}

\newtheorem{remark}[thm]{Remark}

\title[Effective length-projection bounds]{Effective length-projection bounds for hyperbolic 3-manifolds diffeomorphic to $S\times\mathbb{R}$}
\author{Gabriele Viaggi}

\begin{document}

\begin{abstract}
We give a formula with explicit constants relating the subsurface projection $d_Y(\nu^-,\nu^+)$ of the end invariants $\nu^-,\nu^+$ of a hyperbolic 3-manifold $Q$ diffeomorphic to $S\times\mathbb{R}$ and the length of the geodesic representative in $Q$ of the multicurve $\partial Y$. This makes effective and computable the large projections versus short curves relation proved by Minsky. We give an application to closed hyperbolic 3-manifolds fibering over the circle providing a geometric analog of the uniform projection bound in fibered faces of Minsky and Taylor.
\end{abstract}

\maketitle

\section{Introduction}

The subject of this article is the relation between the lengths of closed geodesics in a hyperbolic 3-manifold diffeomorphic to $S\times\mb{R}$, where $S$ is a closed orientable surface of genus at least 2, and the combinatorics of the graph whose vertices are the essential simple closed curves on $S$ and where edges correspond to the disjointness relation, the so-called {\em curve graph} of $S$ which we denote by $\mc{C}(S)$. 

The link between these two seemingly unrelated worlds is a major discovery due to Minsky \cite{M00,M01} and is part of a much deeper connection that led to the solution of the Ending Lamination Conjecture by Minsky \cite{M10} and Brock, Canary, and Minsky \cite{BrockCanaryMinsky:ELC2}. Our goal is to make a specific part of this correspondence {\em effective} as we will explain. Towards the same direction there has been a lot of effort recently, see for example the works of Aougab, Patel, and Taylor \cite{APT}, Futer and Schleimer \cite{FS}, and Futer, Purcell, and Schleimer \cite{FPS19,FPS22,FPS22b}, but the program of effectivizing the relation between geometry of hyperbolic 3-manifolds and the geometry of the curve graph is far from being complete.

By the solution of the Ending Lamination Conjecture \cite{M10,BrockCanaryMinsky:ELC2}, we know that hyperbolic manifolds $Q$ diffeomorphic to $S\times\mb{R}$ are classified by the so-called {\em end invariants} which are pairs $\nu^-,\nu^+$ each of which is a combination of either (minimal filling) laminations or finite area hyperbolic metrics on proper essential subsurfaces $W\subset S$. The natural problem arising from this description is to read off geometric information about $Q$ in terms of its end invariants. 

In this article we focus on the length of closed geodesics in $Q$. Each end invariant $\nu^-,\nu^+$ determines a trace on every proper essential subsurface $Y\subset S$, the so-called {\em subsurface projection} of $\nu^-,\nu^+$, which is a subset of the curve graph of the subsurface $\mc{C}(Y)$ and is denoted by $\pi_Y(\nu^-),\pi_Y(\nu^+)$. Minsky proves in \cite{M00} that if $d_Y(\nu^-,\nu^+):={\rm diam}_{\mc{C}(Y)}(\pi_Y(\nu^-)\cup\pi_Y(\nu^+))$ is {\em large} then $\partial Y$ is isotopic to a {\em short} geodesic link in $Q$. The works \cite{M10,BrockCanaryMinsky:ELC2} further refine the statement by giving formulas for the length of the geodesic representative of a curve $\gamma\subset S$ in terms of subsurface projections {\em up to multiplicative constants}. However, the constants appearing in \cite{M00,M01,M10,BrockCanaryMinsky:ELC2} arise from various compactness arguments and are not computable. Additionally, it is not at all clear how they change with the topology of the surface. Addressing these questions is our main novel contribution. We prove the following quantitative relation between lengths and projections.

\medskip

\begin{thmA}
\label{thm:main1}
There exist constants $a,b,c>0$ such that the following holds. Let $S$ be a closed orientable surface of genus at least 2. Consider a hyperbolic 3-manifold $Q$ diffeomorphic to $S\times\mb{R}$ with end invariants $\nu^-,\nu^+$. Then, for every proper essential connected non-annular subsurface $Y\subset S$ such that both $\nu^-,\nu^+$ intersect $Y$ and none of the components of $\partial Y$ is parabolic in $Q$ we have the following. If
\[
d_Y(\nu^-,\nu^+)\ge a|\chi(Y)|^{345}+b\log|\chi(S)|,
\]
then $\gamma_Y\times\{0\}$ is isotopic to a geodesic link in $Q$ of length 
\[
\ell_Q(\partial Y)\le\frac{c|\chi(Y)|^{248}}{d_Y(\nu^-,\nu^+)-b\log|\chi(S)|}.
\]
\end{thmA}

\medskip

The constants $a,b,c$ are explicit and computable, in fact, even if we did not try to optimize their values, one can choose $a=2^{1095}/\ep_0^{200},b=1040\log_2(2^{385}/\ep_0^{60}),c=2^{331}/\ep_0^{160}$ where $\ep_0$ is a suitably chosen Margulis constant independent of everything (and again, explicitly computable). Note also that, by the definition we adopt, in some cases the subsurface projections of the end invariants are well-defined up to additive errors proportional to $|\chi(S)|$, so the additive error at the denominator is somewhat inevitable (see Section \ref{subsec:from ends to laminations} where the issue is discussed, in particular Definition \ref{dfn:intersection} and Remark \ref{rmk:additive error}).

Extending the result from closed surfaces to surfaces with punctures should not be a problem, but it would require some additional bookkeeping and perhaps some modification to the values of the constants. In order to streamline the exposition we stick to closed surfaces. 

More important is the restriction to non-annular subsurfaces as it is reflected in geometric features of the Margulis tubes of the geodesic representatives of $\partial Y$ in the hyperbolic manifold $Q$. Heuristically speaking, a large non-annular projection forces the boundaries of these Margulis tubes to be (quantitatively) long and skinny. Thus, having to cross the boundary tube the long way, the so-called standard meridians of the core geodesic of the tube are (quantitatively) long (a key fact for our techniques to work). Instead, if we only had large annular projections, the boundaries of these Margulis tubes could be of uniformly bounded shape and the reason why the standard meridians of the core geodesics are (quantitatively) long is different, namely, they twist a lot around the boundary torus. For a more detailed explanation we refer the reader to \cite[Section 2]{FSVb}.

As it is customary, a result about hyperbolic manifolds diffeomorphic to $S\times\mb{R}$ also gives some insight about the geometry of closed hyperbolic 3-manifolds fibering over the circle. This is also the case for Theorem \ref{thm:main1} as we now discuss.

The existence of hyperbolic 3-manifolds $M$ fibering over the circle is a surprising discovery of Jørgensen \cite{J}. A subsequent breakthrough of Thurston \cite{Thu2} showed that, in fact, every atoroidal 3-manifold fibering over the circle admits a hyperbolic metric. These manifolds can be equivalently described as mapping tori of so-called {\em pseudo-Anosov diffeomorphisms} $f:S\to S$ of closed orientable surfaces $S$, that is $M=S\times[0,1]/(x,0)\sim(f(x),1)$. The covering of $M$ corresponding to $\pi_1(S\times\{0\})$ is a hyperbolic 3-manifold diffeomorphic to $S\times\mb{R}$. Its end invariants are precisely the two invariant laminations of the pseudo-Anosov diffeomorphism.

Thurston also realized \cite{ThuMemo} that often a closed 3-manifold $M$ fibers over the circle in many different ways that are organized by the so-called {\em Thurston norm} on ${\rm H}_2(M;\mb{R})$ in the following way. The unit ball for such norm is a finite sided polyhedron with a distinguished set of top dimensional open faces, the so-called {\em fibered faces}, denoted by ${\bf F}$. Thurston proves that the homology classes of the fibers of the fibrations of $M$ over the circle are exactly the intersection of the primitive elements of the integer lattice ${\rm H}_2(M;\mb{Z})$ with the positive cone $ (0,\infty)\times{\bf F}$. Each such fiber $F\in(0,\infty)\times{\bf F}$ provides a description of $M$ as a mapping torus of a pseudo-Anosov diffeomorphism $f_F:F\to F$ with invariant laminations $\nu^-_F,\nu^+_F$. It is natural to ask whether the subsurface projections $d_Y(\nu_F^-,\nu_F^+)$ for $F\in(0,\infty)\times{\bf F}$ can be simultaneously and uniformly bounded in terms of $M$. 

This was done by Minsky and Taylor \cite{MT17,MT24} from a purely topological perspective. They discovered, among other things, two important phenomena (see \cite[Theorems 1.1 and 1.2]{MT24}): All subsurface projections can be explicitly and uniformly bounded in terms of the number of tetrahedra in a so-called {\em veering triangulation} of $M$. Subsurfaces of fibers $Y\subset F$ whose projection coefficients $d_Y(\nu_F^-,\nu_F^+)$ are large compared to $|\chi(S)|$ for a fixed fiber $S$ in the same component of the fibered cone correspond to embedded subsurfaces of $S$ with the same projection coefficients. Combining Theorem \ref{thm:main1} with their results, we can therefore conclude the following geometric counterpart of their bound. The main new contribution in the above is the explicit formula and the dependence on the injectivity radius. 

\medskip

\begin{thmA}
\label{thm:main2}
There exists a constant $k>0$ such that the following holds. Let $M$ be a closed orientable hyperbolic 3-manifold fibering over the circle. Let $(0,\infty)\times{\rm\bf F}_0$ be a component of the fibered cone $(0,\infty)\times{\rm\bf F}$ for $M$ and denote by $S$ a genus minimizing surface in ${\rm\bf F}_0\times(0,\infty)$. For every fiber $F\subset M$ whose homology class lies in $(0,\infty)\times{\bf F}_0$ and for every proper essential non-annular subsurface $Y\subset F$ we have
\[
d_Y(\nu_F^-,\nu_F^+)\le\frac{k|\chi(S)|^{346}}{{\rm inj}(M)}.
\]
\end{thmA}

\medskip

Next we briefly describe what goes into the proof of Theorem \ref{thm:main1}.

\subsection*{On the proof}
Our strategy is inspired by \cite{FSVa,FSVb}. We start by observing that $S\times\mb{R}$ can be seen as a {\em Dehn filling} of the drilled manifold $M=S\times\mb{R}-\gamma_Y\times\{0\}$ where $\gamma_Y$ is the multicurve obtained from $\partial Y$ by collapsing parallel components to single ones. At the level of hyperbolic metrics, we would like to say that the metric $Q$ with end invariants $\nu^-,\nu^+$ can be obtained as a hyperbolic Dehn filling of the hyperbolic metric on $M$ with rank two cusps at $\gamma_Y\times\{0\}$ and end invariants $\nu^-,\nu^+$. More quantitatively, we would like to say that the so-called {\em normalized length} of the meridians of the components of $\gamma_Y$ is bounded from below by an explicit formula as the one in Theorem \ref{thm:main1}. This is the main challenge. In order to deal with it we use two ingredients. First, a variation of a result of Aougab, Patel, and Taylor \cite{APT} combined with covering arguments as in \cite{FSVa,FSVb} will tell us that if there are two simple closed curves $\delta^-,\delta^+\subset Y$ whose geodesic representatives in $M$ have length at most $L$ then the length of the meridians of the components of $\gamma_Y$ are at least $d_{\mc{C}(Y)}(\delta^-,\delta^+)/de^{3L}$ for some universal constant $d$ (Theorem \ref{thm:effcusp} and Proposition \ref{pro:covering2}). Thus, we are left with the problem of producing from the end invariants $\nu^-,\nu^+$ simple closed curves $\delta^-,\delta^+\subset Y$ representing $\pi_Y(\nu^-),\pi_Y(\nu^+)$ of explicit moderate length. This is the core of the argument and the second ingredient (Theorem \ref{thm:effeff}). In Minsky's work \cite{M00} such curves are produced using the Efficiency of Pleated Surfaces \cite[Theorem 3.5]{M00} and our approach builds upon a similar principle. However, rather than relying on the Uniform Injectivity Theorem \cite[Theorem 3.2]{M00}, we follow a different, more topological, path which is allowed by our setup and gives us effective length bounds for the surgeries $\delta^-,\delta^+$ of $\nu^-,\nu^+$.

\subsection*{Organization of the article}
The paper is structured as follows.

Sections \ref{sec:2}, \ref{sec:3}, \ref{sec:4} develop the necessary background on hyperbolic surfaces, curve graphs, and hyperbolic 3-manifolds as is needed in the proof of Theorem \ref{thm:main1} which will be discussed only in Section \ref{sec:proof}. The proof in Section \ref{sec:proof} is organized in several steps and requires two crucial ingredients: An effective efficiency result (Theorem \ref{thm:effeff}) and an effective convex core width result (Theorem \ref{thm:effcusp}). These are developed in Sections \ref{sec:effeff} (which might be of independent interest) and \ref{sec:effcusp}.   

\subsection*{Acknowledgements}
I thank the anonymous referee for the careful reading and the very helpful suggestions. I would like to thank Alessandro Sisto and Peter Feller for their generous feedback on the first draft of this paper.

\section{Geometry and topology of hyperbolic surfaces}
\label{sec:2}

In this section we review some of the background on hyperbolic surfaces. There are essentially two properties that we need: 
\begin{itemize}
    \item{Existence of free subgroups generated by moderate-length loops (Lemma \ref{lem:short gen}). This is a standard consequence of the Gauss-Bonnet formula expressing the area of a hyperbolic surface in terms of its topology.}
    \item{Effective recurrence (Lemma \ref{lem:recurrence}) and quasi-geodesic closing (Lemma \ref{lem:quasi-geodesic}). The first, which is again a consequence of Gauss-Bonnet, tells us that a geodesic in the unit tangent bundle of a hyperbolic surface comes often close to itself. The second is a consequence of the local-to-global property of quasi-geodesics in hyperbolic space.}
\end{itemize}

As we wish to obtain explicit bounds, we emphasize the explicit dependence on the topology of the surfaces of all the constants we encounter in the above facts.

\subsection{Thick-thin decomposition}
We begin by recalling a basic structural result of hyperbolic manifolds, namely the thick-thin decomposition. In the case of hyperbolic surfaces the thin parts have a particularly simple description. 

\begin{dfn}[Standard Collar and Cuspidal Part]
Let $\ep_0>0$ be a Margulis constant for dimension 2. If $(Z,\sigma)$ is a hyperbolic surface, then, by the thick-thin decomposition (see \cite[Chapter 4]{Martelli}), every connected component of the {\em $\ep_0$-thin region}
\[
(Z,\sigma)_{\ep_0}:=\{x\in Z\,|\,{\rm inj}_x(Z)<\ep_0\}
\]
is either 
\begin{itemize}
    \item{A {\em $\ep_0$-standard collar} of a simple closed geodesic $\gamma\subset Z$ of length $\ell_\sigma(\gamma)<2\ep_0$, a component which we denote by ${\rm\bf collar}(\gamma,\ep_0)$. Topologically, it is a tubular neighborhood of $\gamma$.}
    \item{A {\em $\ep_0$-cuspidal neighborhood} of an end of $Z$ (a cusp) which we denote by ${\rm\bf cusp}(\gamma,\ep_0)$. Topologically, ${\rm\bf cusp}(\gamma,\ep_0)$ is a punctured disk and $\gamma$ denotes the homotopy class of the simple closed curve $\gamma$ encircling the puncture clockwise.}
\end{itemize}   

The complement of the $\ep_0$-thin region is the {\em $\ep_0$-thick part}.
\end{dfn}

\subsection{Moderate-length curves}
A standard application of the Gauss-Bonnet formula ${\rm Area}(Z,\sigma)=2\pi|\chi(Z)|$ allows us to find curves of uniformly bounded length everywhere on a finite area hyperbolic surface $(Z,\sigma)$. The next lemma, besides providing explicit formulas for the length bounds, gives two different refinements of this fact that will be important for us. The first one upgrades a simple closed geodesic loop of moderate length to a pair of loops of moderate length generating a free subgroup provided that we are in the thick part of the surface (this will play a role in Section \ref{sec:effeff}). The second instead promotes a single moderate-length geodesic to an entire moderate-length pants decomposition, a so-called {\em Bers pants decomposition}.

\begin{lem}
\label{lem:short gen}
Let $(Z,\sigma)$ be a complete finite area orientable hyperbolic surface.
\begin{enumerate}
    \item{For every $x\in Z$ there exists a simple geodesic loop $\gamma$ based at $x$ of length
    \[
    \ell(\gamma)\le 2\cdot{\rm arccosh}(|\chi(Z)|+1).
    \]
    }
    \item{Let $\ep_0<{\rm arcsinh}(1/4)$ be a Margulis constant for dimension 2. For every $x\in Z$ in the $\ep_0$-thick part there exist a pair of simple geodesic loops $\gamma,\gamma'$ based at $x$ of length
    \[
    \ell(\gamma),\ell(\gamma')\le 2\log\left(\frac{256\pi^2|\chi(Z)|^4}{\ep_0^2}\right).
    \]
    that generate  a subgroup $\langle\gamma,\gamma'\rangle<\pi_1(Z,x)$ which is free of rank 2.}
    \item{There exists a pants decomposition of $Z$, called a {\rm Bers pants decomposition}, where each component is a geodesic of length at most $2\pi|\chi(Z)|$.}
\end{enumerate}
\end{lem}

\begin{remark}
In the lemma and throughout the text, we use the terminology {\em geodesic loop based at some point} to denote a loop that is geodesic except perhaps at the basepoint where the initial and terminal velocities are allowed to be non-collinear. This has to be distinguished from what we will call a {\em (simple) closed geodesic} that is an (injective) geodesic loop where the initial and terminal velocities are aligned. 
\end{remark}

\begin{proof}
{\bf Property (1)}. Let $\gamma$ be the shortest loop around $x$. By negative curvature, we have $\ell(\gamma)=2\cdot{\rm inj}_x(Z)$ and, hence, as the ball centered at $x$ of radius equal to the injectivity radius is isometric to a ball in $\mb{H}^2$ of the same radius, we have
\[
{\rm Area}(B(x,\ell(\gamma)/2))=2\pi(\cosh(\ell(\gamma)/2)-1).
\]
The area of the embedded ball is smaller than the area of the entire surface which, by Gauss-Bonnet, equals $2\pi|\chi(Z)|$. Thus
$\ell(\gamma)\le 2\cdot{\rm arccosh}(|\chi(Z)|+1)$.

{\bf Property (2)}. Let $\gamma$ be the shortest loop around $x$ as in Property (1). There are two cases: Either $\gamma$ is parabolic or not.

{\slshape Case $\gamma$ parabolic}. Let $p:A\to Z$ be the covering of $Z$ corresponding to the subgroup $\langle\gamma\rangle<\pi_1(Z,x)$. The surface $A$ is a cusp isometric to $(\mb{S}^1\times\mb{R},e^{-2r}{\rm d}l^2+{\rm d}r^2)$ where $\mb{S}^1=[0,1]/0\sim1$. The curve $\gamma$ lifts to a simple closed geodesic loop $\overline{\gamma}\subset A$ based at a point $\overline{x}=(u,R)\in\mb{S}^1\times\mb{R}$. The injectivity radius at a point $\overline{y}=(l,r)\in\mb{S}^1\times\mb{R}$ is half of the length of the geodesic loop homotopic to $\gamma$ based at $\overline{y}$. By basic hyperbolic geometry, it is given by 
\[
{\rm inj}_{\overline{y}}(A)={\rm arcsinh}\left(1/2e^r\right).
\]
In particular, the radius $R$ of $\overline{x}$ satisfies
\[
\ell(\gamma)/2={\rm inj}_{\overline{x}}(A)={\rm arcsinh}\left(1/2e^{R}\right)
\]
while the radius $R_0$ at which the injectivity radius is $\ep_0$ is given by
\[
\ep_0={\rm arcsinh}\left(1/2e^{R_0}\right).
\]

As we assumed $\ell(\gamma)\ge2\ep_0$ we have $R\le R_0$. 

Consider the area of the region $\mb{S}^1\times[R_0,+\infty)$. A simple integration yields
\[
{\rm Area}(\mb{S}^1\times[R_0,+\infty))=e^{-R_0}=2\sinh(\ep_0)<2.
\]

Note also that ${\rm Area}(\mb{S}^1\times[R_1,+\infty))=2\pi|\chi(Z)|+2$ for the radius
\[
R_1=-\log(2\pi|\chi(Z)|+2).
\]

Since ${\rm Area}(\mb{S}^1\times[R_1,+\infty))-{\rm Area}(\mb{S}^1\times[R_0,+\infty))>{\rm Area}(Z,\sigma)$, we deduce that the restriction of $p$ to the region $\mb{S}^1\times[R_1,R_0]$ is not injective. Thus, there are two distinct points $\overline{y}=(l,r),\overline{y}'=(l',r)$ (with $R_1\le r\le R_0$) such that $p(\overline{y})=p(\overline{y}')$. As above, the path $\alpha=p([l,l']\times\{r\})$ is an essential loop on $Z$ based at the projections $p(\overline{y})=p(\overline{y}')$ which is not homotopic to $\gamma$. Its geodesic representative with fixed endpoints $\overline{\alpha}$ has length 
\[
\ell(\overline{\alpha})=2{\rm arcsinh}(d_{\mb{S}^1}(l,l')/2e^{r})\le 2{\rm arcsinh}(1/2e^{R_1})=2{\rm arcsinh}\left(\pi|\chi(Z)|+1\right).
\]

Pre- and post- concatenating $\overline{\alpha}$ with $p((\{u\}\times[R,R_0])\star([u,l]\times\{R_0\})\star(\{l\}\times[R_0,R_1]))$ and its inverse (recall that $\overline{x}=(u,R)$) we obtain a loop $\gamma'$ based at $x$ and freely homotopic to $\alpha$. Its length is bounded by
\begin{align*}
\ell(\gamma')-2\ep_0 &\le 4R_0-2R_1-2R+\ell(\overline{\alpha})\\
&\le 4\log\left(\frac{1}{2\sinh(\ep_0)}\right)+2\log(2\pi|\chi(Z)|+2)-2\log\left(\frac{1}{2\sinh(\ell(\gamma)/2)}\right)\\
&+2{\rm arcsinh}(\pi|\chi(Z)|+1).
\end{align*}
As ${\rm arcsinh}(x)\le\log(4x)$ for $x\ge 1$, $\pi|\chi(Z)|+1\le 2\pi|\chi(Z)|$, and $\sinh(\ell(\gamma)/2)\le|\chi(Z)|$ (a consequence of Property (1)) we obtain the inequality
\[
\le 2\log\left(\frac{16\pi^2|\chi(Z)|^3}{\sinh(\ep_0)^2}\right)\le 2\log\left(\frac{16\pi^2|\chi(Z)|^3}{\ep_0^2}\right).
\]

As $\gamma'$ is not homotopic to $\gamma$, the subgroup $\langle\gamma,\gamma'\rangle$ is a free non-abelian subgroup of rank 2 (by the structure of subgroups of surface groups $\pi_1(Z)$). This concludes the first case.

{\slshape Case $\gamma$ non-parabolic}. Consider $\gamma^*\subset Z$, the geodesic representative of $\gamma$. Let $p:A\to Z$ be the covering of $Z$ corresponding to the subgroup $\langle\gamma\rangle<\pi_1(Z,x)$. The surface $A$ is a hyperbolic annulus with core geodesic given by a lift $\overline{\gamma}^*$ of $\gamma^*$. The curve $\gamma$ lifts to a simple closed geodesic loop $\overline{\gamma}\subset A$ based at a point $\overline{x}=(u,R)$. We now recall that in Fermi coordinates for $\overline{\gamma}^*$, the metric of $A$ is written as ${\rm d}r^2+\cosh(r)^2{\rm d}l^2$ with $l\in\mb{S}^1_{\ell(\gamma^*)}=[0,\ell(\gamma^*)]/0\sim\ell(\gamma^*)$ and $r\in\mb{R}$. The injectivity radius at a point $\overline{y}=(l,r)\in\mb{S}^1_{\ell(\gamma^*)}\times\mb{R}$ is half of the length of the geodesic loop homotopic to $\gamma$ based at $\overline{y}$. By basic hyperbolic geometry, it is given by
\[
{\rm inj}_{\overline{y}}(A)={\rm arcsinh}\left(\cosh(r)\sinh(\ell(\gamma^*)/2)\right).
\]
In particular, the radius $R$ corresponding to $\overline{x}$ solves
\[
\ell(\gamma)/2={\rm inj}_{\overline{x}}(A)={\rm arcsinh}\left(\cosh(R)\sinh(\ell(\gamma^*)/2)\right).
\]

We need to distinguish two sub-cases: Either $\ell(\gamma^*)<2\ep_0$ or not.

{\slshape Case $\ell(\gamma^*)<2\ep_0$}. The radius $R_0$ at which the injectivity radius is $\ep_0$ satisfies
\[
\ep_0={\rm arcsinh}\left(\cosh(R_0)\sinh(\ell(\gamma^*)/2)\right).
\]
As we are assuming $\ell(\gamma^*)<2\ep_0$ we can also conclude that $R<R_0$. 

Similarly to the parabolic case, we consider areas. We have
\[
{\rm Area}\left(\mb{S}^1_{\ell(\gamma^*)}\times[0,r]\right)=\ell(\gamma^*)\sinh(r).
\]
Note that for $r=R_0$ we have
\begin{align*}
{\rm Area}\left(\mb{S}^1_{\ell(\gamma^*)}\times[0,R_0]\right) &=\ell(\gamma^*)\sinh\left({\rm arccosh}\left(\frac{\sinh(\ep_0)}{\sinh(\ell(\gamma^*)/2)}\right)\right)\\
 &\le \sinh(\ep_0)\frac{\ell(\gamma^*)}{\sinh(\ell(\gamma^*)/2)}\le 2\sinh(\ep_0)<2
\end{align*}
where we used $\sinh({\rm arccosh}(x))=\sqrt{x^2-1}\le x$ and $x\le\sinh(x)$. Note also that ${\rm Area}\left(\mb{S}^1_{\ell(\gamma^*)}\times[0,R_1]\right)=2\pi|\chi(Z)|+2$ for the radius 
\[
R_1={\rm arcsinh}\left(\frac{2\pi|\chi(Z)|+2}{\ell(\gamma^*)}\right).
\] 
In particular ${\rm Area}\left(\mb{S}^1_{\ell(\gamma^*)}\times[0,R_1]\right)-{\rm Area}\left(\mb{S}^1_{\ell(\gamma^*)}\times[0,R_0]\right)>2\pi|\chi(Z)|$. 

As a consequence, the restriction of $p$ to $\mb{S}^1_{\ell(\gamma^*)}\times[R_0,R_1]$ is not injective. Thus there are two distinct points $\overline{y}=(l,r),\overline{y}'=(l',r)$ (with $r\in[R_0,R_1]$) such that $p(\overline{y})=p(\overline{y}')$. By covering theory, the path $\alpha=p([l,l']\times\{r\})$ is an essential loop on $Z$ based at the projections $p(\overline{y})=p(\overline{y}')$ which is not homotopic to $\gamma$. By basic hyperbolic geometry, its geodesic representative $\overline{\alpha}$ with fixed endpoints has length
\[
\ell(\overline{\alpha})=2{\rm arcsinh}\left(\sinh(d_{\mb{S}^1_{\ell(\gamma^*)}}(l,l'))\cosh(r)\right)\le2{\rm arcsinh}\left(\sinh(\ell(\gamma^*)/2)\cosh(R_1)\right)\\
\]
using $\cosh(R_1)\le 2\sinh(R_1)$ and the formula for $R_1$ we can continue with
\[
\le2{\rm arcsinh}\left(\frac{\sinh(\ell(\gamma^*)/2)}{\ell(\gamma^*)/2}(2\pi|\chi(Z)|+2)\right).
\]

Pre- and post- concatenating $\overline{\alpha}$ with $p(([u,l]\times\{R\})\star(\{l\}\times[R,r]))$ and its inverse (recall that $\overline{x}=(u,R)$) we obtain a loop $\gamma'$ based at $x$ freely homotopic to $\alpha$. Its length is bounded by
\begin{align*}
\ell(\gamma')-2\ep &\le 2R_1-2R+\ell(\overline{\alpha})\\
&=2{\rm arcsinh}\left(\frac{\pi|\chi(Z)|+1}{\ell(\gamma^*)/2}\right)-2{\rm arccosh}\left(\frac{\sinh(\ell(\gamma)/2)}{\sinh(\ell(\gamma^*)/2)}\right)\\
&+2{\rm arcsinh}\left(\frac{\sinh(\ell(\gamma^*)/2)}{\ell(\gamma^*)/2}(2\pi|\chi(Z)|+2)\right).
\end{align*}
Using $\log(x)\le{\rm arccosh}(x)$ and ${\rm arcsinh}(x)\le\log(4x)$ for $x\ge 1$, and $\pi|\chi(Z)|+1\le 2\pi|\chi(Z)|$, we get
\[
\le 2\log\left(\frac{64\pi^2|\chi(Z)|^2}{\sinh(\ell(\gamma)/2)}\frac{\sinh(\ell(\gamma^*)/2)^2}{\ell(\gamma^*)^2/4}\right).
\]
As $\sinh(x)/x$ has values in $(1,2)$ on $(0,\ep_0)$ and $\ell(\gamma)\ge 2\ep_0$ we obtain the inequality
\[
\le 2\log\left(\frac{256\pi^2|\chi(Z)|^2}{\sinh(2\ep_0)}\right)\le2\log\left(\frac{128\pi^2|\chi(Z)|^2}{\ep_0}\right).
\]
Thus 
\[
\ell(\gamma')\le 2\ep+2\log\left(\frac{128\pi^2|\chi(Z)|^2}{\ep_0}\right)\le 2\log\left(\frac{128\pi^2|\chi(Z)|^2}{\ep_0^2}\right).
\]

Since $\gamma'$ is not homotopic to $\gamma$, the subgroup $\langle\gamma,\gamma'\rangle$ is a free non-abelian subgroup of rank 2 (by the structure of subgroups of surface groups $\pi_1(Z)$).

{\slshape Case $\ell(\gamma^*)\ge 2\ep_0$}. We have ${\rm Area}\left(\mb{S}^1_{\ell(\gamma^*)}\times[0,r]\right)=\ell(\gamma^*)\sinh(r)\ge 2\ep_0\sinh(r)$. Thus ${\rm Area}\left(\mb{S}^1_{\ell(\gamma^*)}\times[0,R_1]\right)\ge 2\pi|\chi(Z)|$ for the radius 
\[
R_1={\rm arcsinh}(\pi|\chi(Z)|/\ep_0).
\]

As before, comparing with ${\rm Area}(Z,\sigma)$ we deduce that the restriction of $p$ to $\mb{S}^1_{\ell(\gamma^*)}\times[0,R_1]$ is not injective. Thus there are two distinct points $\overline{y}=(l,r),\overline{y}'=(l',r)$ such that $p(\overline{y})=p(\overline{y}')$. By covering theory, the path $\alpha=p(([l,l']\times\{r\}))$ is an essential loop on $Z$ based at the projections $p(\overline{y})=p(\overline{y}')$ which is not homotopic to $\gamma$. By basic hyperbolic geometry, its geodesic representative $\overline{\alpha}$ with fixed endpoints has length
\[
\ell(\overline{\alpha})=2{\rm arcsinh}\left(\sinh\left(d_{\mb{S}^1_{\ell(\gamma^*)}}(l,l')\right)\cosh(r)\right)\le2{\rm arcsinh}\left(\sinh(\ell(\gamma^*)/2)\cosh(R_1)\right)
\]
using $\cosh(R_1)\le 2\sinh(R_1)$ and the formula for $R_1$ we can continue with
\[
\le2{\rm arcsinh}\left(\sinh(\ell(\gamma^*)/2)\frac{2\pi|\chi(Z)|}{\ep_0}\right).
\]

Pre- and post- concatenating $\overline{\alpha}$ with $p(\overline{(\{u\}\times[0,R])}\star([u,l]\times\{0\})\star(\{l\}\times[0,r]))$ and its inverse (recall that $\overline{x}=(u,R)$) we obtain a loop based at $x$ and freely homotopic to $\alpha$. Its length is bounded by
\begin{align*}
&\ell(\gamma')\le 2R_1+2R+\ell(\gamma^*)+\ell(\overline{\alpha})\le 2R_1+2R+\ell(\gamma)+\ell(\overline{\alpha})\\
&\le2{\rm arcsinh}\left(\frac{\pi|\chi(Z)|}{\ep_0}\right)+2{\rm arccosh}\left(\frac{\sinh(\ell(\gamma)/2)}{\sinh(\ell(\gamma^*)/2)}\right)+2{\rm arccosh}(|\chi(Z)|+1)\\
&+2{\rm arcsinh}\left(\sinh(\ell(\gamma^*)/2)\frac{2\pi|\chi(Z)|}{\ep_0}\right).
\end{align*}
As $\le{\rm arccosh}(x)\le\log(2x)$ for $x>1$, ${\rm arcsinh}(x)\le\log(4x)$ for $x\ge 1$, we obtain 
\[
\le2\log\left(256\frac{\pi^2|\chi(Z)|^3}{\ep_0^2}\sinh(\ell(\gamma)/2)\right).
\]
Lastly, we recall that $\sinh(\ell(\gamma)/2)\le|\chi(Z)|$ (a consequence of Property (1)) and get
\[
\le2\log\left(256\frac{\pi^2|\chi(Z)|^4}{\ep_0^2}\right).
\]

Since $\gamma'$ is not homotopic to $\gamma$, the subgroup $\langle\gamma,\gamma'\rangle$ is a free non-abelian subgroup of rank 2 (by the structure of subgroups of surface groups $\pi_1(Z)$). 

This concludes the non-parabolic case and the proof of Property (2).

{\bf Property (3)}. See \cite{Parlier}.
\end{proof}

\subsection{Recurrence and quasi-geodesic closing}
Lastly, we discuss a recurrence property of the geodesic flow on a finite area hyperbolic surface and how to use it to approximate geodesic segments with closed geodesics. This will play a role in the proof of Theorem \ref{thm:effeff}. The following lemma says in a quantitative way that geodesic segments in the thick part come often close to themselves.  

\begin{lem}
\label{lem:recurrence}
Fix $\ep_0>0$. For every $\ep<\ep_0/2$ denote by ${\rm vol}(\ep)$ the volume of a ball of radius $\ep$ in $T^1\mb{H}^2$. Let $Z$ be a connected orientable complete hyperbolic surface of finite area. Consider a geodesic segment (parameterized by unit speed) $\tau:[0,T]\to Z$ such that ${\rm inj}_{\tau(t)}(Z)\ge\ep_0$ for every $t$. For every sequence of times $0\le t_1<\cdots<t_m\le T$ such that $t_j-t_{j-1}\ge 1$. Denote by $V:=\{v_j:=\tau'(t_j)\}\subset T^1Z$ the corresponding set of velocities. There exists $\{w_1,\cdots,w_n\}\subset V$ with 
\[
n=\left\lfloor\frac{{\rm vol}(\ep/2)}{2\pi|\chi(Z)|}m\right\rfloor
\]
such that $d_{T^1Z}(w_i,w_j)<\ep$ for every $i,j\le n$.    
\end{lem}
  
\begin{proof}
Denote by $B_j$ the ball of radius $\ep/2$ around $v_j=\tau'(t_j)$ in $T^1Z$. As the injectivity radius at $\tau(t_j)$ is at least $\ep_0$ and $\ep/2<\ep_0$, the ball $B_j$ is isometric to a ball of radius $\ep/2$ in $T^1\mb{H}^2$. Let $\mathds{1}_{B_j}$ be the characteristic function of $B_j$. We have
\begin{align*}
{\rm vol}(\ep/2)m &=\sum_{j\le m}{{\rm vol{(B_j)}}}\\
 &=\sum_{j\le m}{\int_{T^1Z}{\mathds{1}_{B_j}(v)}{\rm dvol}_{T^1Z}(v)}\\
 &=\int_{T^1Z}{\sum_{j\le m}{\mathds{1}_{B_j}(v)}{\rm dvol}_{T^1Z}(v)}=\int_{T^1Z}{n(v){\rm dvol}_{T^1Z}(v)}
\end{align*}
where $n(v)$ is the number of balls $B_j$ that contain $v\in T^1Z$. If we set $n={\rm max}_{v\in T^1Z}\{n(v)\}$ we get
\[
{\rm vol}(\ep/2)m\le n\cdot{\rm vol}(T^1Z)=2\pi|\chi(Z)|\Longrightarrow n\ge\frac{{\rm vol}(\ep/2)}{2\pi|\chi(Z)|}m.
\]
If $v\in T^1Z$ is a point realizing $n$ and $w_1,\cdots, w_n$ are the centers of the balls $B_j$ containing $v$, then, by the triangle inequality, $d_{T^1Z}(w_i,w_j)\le\ep$ for every $i,j\le n$.
\end{proof} 

Once we have a geodesic segment that comes close to itself in the unit tangent bundle, we can close it and obtain a closed geodesic that approximates it. The tool to do so is the following well-known quasi-geodesic closing lemma.

\begin{lem}
\label{lem:quasi-geodesic}
For every $\ep\in(0,\pi/6]$ there exists an explicit constant $T=T(\ep)$ such that the following holds. Let $M$ be a complete hyperbolic $k$-manifold. Consider a concatenation of geodesic segments $\eta\star\mu\subset M$ such that $\ell(\eta)\ge T,\ell(\mu)\le\ep$ and the terminal (resp. initial) velocity of $\eta$ forms an angle of at most $\ep$ with the initial (resp. terminal) velocity of $\mu$. Then $\eta\star\mu$ is homotopic to a closed geodesic. 
\end{lem}

Numerically, we can choose 
\[
T(\pi/6)=3072\log(2)\log_2(148)+3280\log(2)+384<200000.
\]
This value can be computed by a careful bookkeeping following the proofs of \cite[Theorem 1.2]{CDP} and \cite[Theorem 1.4]{CDP} using as input $\lambda=1,k=\log(4),\delta=\log(2)$.

\begin{proof}
Let $\tau$ be the geodesic loop in the same homotopy class with fixed endpoints as $\eta\star\mu$. The composition $\eta\star\mu\star\overline{\tau}$ is nullhomotopic in $M$ (where $\overline{\tau}$ is $\tau$ with the inverse parameterization) and, hence, it lifts to a geodesic triangle in $\mb{H}^k$, the universal cover of $M$. By the triangle inequality $\ell(\tau)\ge\ell(\eta)-\ell(\mu)\ge T-\ep$. Additionally, the angle between the terminal (resp. initial) velocity of $\tau$  and the terminal (resp. initial) velocity of $\mu$ (resp. $\eta$) is at most $\ep$ (as the sum of the interior angles in the hyperbolic triangle $\eta\mu\tau$ is smaller than $\pi$ and the angle at $\eta\cap\mu$ is larger than $\pi-\ep$ by assumption). By the triangle inequality the angle between the initial and terminal velocities of $\tau$ is at most $3\ep$.

Consider the lift of $\tau:\mb{S}^1\to M$ to the universal cover $\tilde{\tau}:\mb{R}\to\mb{H}^k$. By the discussion above $\tilde{\tau}$ is a piecewise geodesic path made of segments of length $\ell(\tau)\ge T-\ep$ and breaking angles at most $3\ep$. This path is a $(1,\log(4),1)$-local quasi-geodesic (we recall that $\tilde{\tau}$ is a $(L,C,R)$-local quasi-geodesic if for every $t,t'\in\mb{R}$ with $|t-t'|<R$ we have $|t-t'|/L-C\le d_{\mb{H}^k}(\tilde{\tau}(t),\tilde{\tau}(t'))\le L|t-t'|+C$).

In fact, if $t<t'<t+1$ then either $\tilde{\tau}(t),\tilde{\tau}(t')$ lie on the same segment, in which case there is nothing to add, or they lie on two consecutive edges, that is there is exactly a breaking point $t''\in[t,t']$. By the hyperbolic law of cosines
\begin{align*}
\cosh(d_{\mb{H}^k}(\tilde{\tau}(t),\tilde{\tau}(t'))) &=\cosh(t''-t)\cosh(t'-t'')-\cos(\theta)\sinh(t''-t)\sinh(t'-t'')\\
&\ge\cosh(t''-t)\cosh(t'-t'')
\end{align*}
where in the inequality we used the fact that $\theta\ge\pi-3\ep\ge\pi/2$ which implies $\cos(\theta)\le 0$. Using the approximation $e^x/2\le\cosh(x)=(e^x+e^{-x})/2\le e^x$ we get 
\[
e^{d_{\mb{H}^k}(\tilde{\tau}(t),\tilde{\tau}(t'))}\ge e^{t''-t}e^{t'-t''}/4\Rightarrow d_{\mb{H}^k}(\tilde{\tau}(t),\tilde{\tau}(t'))\ge t'-t-\log(4).
\]
By the local-to-global properties of quasi-geodesics (see \cite[Theorem 1.4]{CDP}), if $T$ is larger than some explicitly computable threshold only depending on the local quasi-geodesic constants $1,\log(4),1$ and the hyperbolicity constant of $\mb{H}^k$, then $\tilde{\tau}$ is a global $(c,c)$-quasi-geodesic (for some constant $c$ explicitly computable in terms of the same data).
\end{proof}

\section{Curve graphs}
\label{sec:3}

In this section we recall the definitions of curve graphs and subsurface projections and we discuss some relations between the geometry of these objects and the geometry of hyperbolic surfaces. More precisely:
\begin{itemize}
    \item{The distance-versus-intersection (Lemma \ref{lem:intersec dist}) and the intersection-versus-length (Lemma \ref{lem:length intersection}) bounds. The first is a general tool to bound distances in the curve graph in terms of intersections of simple closed curves. The second, a standard consequence of the Collar Lemma, bounds the number of intersections of two simple closed geodesics on a hyperbolic surface in terms of their lengths.}
    \item{Ending Laminations and convergence to the boundary of the curve graph.}
\end{itemize}

\subsection{Curve graph}
We begin with the following definition.

\begin{dfn}[Curve Graph]
Let $Z$ be a compact connected orientable surface with $\chi(Z)<0$ which is not a 4-holed sphere or a 1-holed torus. The curve graph $\mc{C}(Z)$ is the graph whose vertices are the isotopy classes of essential non-peripheral simple closed curves $\gamma\subset Z$ and there is an edge of length 1 between $\alpha$ and $\beta$ if there are representatives of the isotopy classes that are disjoint. It is a standard fact that the curve graph is connected and hence it has a well-defined path metric which we denote by $d_{\mc{C}(Z)}(\bullet,\bullet)$.

In the special cases of 4-holed spheres or 1-holed tori two essential simple closed curves always intersect at least 2 or 1 times, respectively. For these surfaces we can still construct the curve graph, but edges correspond to simple closed curves that intersect exactly 2 or 1 times, respectively. It is a standard fact that the resulting graph is again connected and has a well-defined path metric.

Lastly, if $Z$ is a 3-holed sphere then $\mc{C}(Z)=\emptyset$.
\end{dfn}

\subsection{Distance and intersections}
It will be important for us to be able to estimate distances in the curve graph. In order to do so, we will use the following useful tool that relates intersections of curves to distances. The original argument due to Hempel \cite[Lemma 2.1]{Hempel:Heegaard} was stated for closed surfaces, but it extends to compact surfaces as explained in \cite[Lemma 1.21]{S}.

\begin{lem}
\label{lem:intersec dist}
We have
\[
d_{\mc{C}(Z)}(\alpha,\beta)\le 2+2\log_2(i(\alpha,\beta))
\]
where $i(\bullet,\bullet)$ is the geometric intersection number between simple closed curves.
\end{lem}

In turn, on a hyperbolic surface, intersections can be bounded in terms of lengths thanks to the Collar Lemma (see for example \cite[Section 7.3.3]{Martelli}). Explicit intersection versus length bounds are stated in the next lemma.

\begin{lem}
\label{lem:length intersection}    
Let $(Z,\sigma)$ be a complete finite area hyperbolic surface. Consider simple closed geodesics $\alpha,\beta\subset Z$ of length $\ell(\alpha),\ell(\beta)$. Then
\[
i(\alpha,\beta)\le \ell(\alpha)e^{\ell(\beta)/2}.
\]
\end{lem}

\begin{proof}
By the Collar Lemma, every simple closed geodesic $\gamma$ on a hyperbolic surface $(Z,\sigma)$ has an embedded normal neighborhood of width 
\[
\log(\coth(\ell(\gamma)/4))=\log\left(\frac{e^{\ell(\gamma)/4}+e^{-\ell(\gamma)/4}}{e^{\ell(\gamma)/4}-e^{-\ell(\gamma)/4}}\right)=\log\left(1+e^{-\ell(\gamma)/2}\right)-\log\left(1-e^{-\ell(\gamma)/2}\right).
\] 
Using the mean value theorem we have $\log(1+x)-\log(1-x)=2x/(1+x_0)$ for some $x_0\in[-x,x]$ and recalling that $|x|<1$ we can further bound from below that quantity by $2x/(1+x_0)\ge x$. Thus, the width of the normal neighborhood is at least $e^{-\ell(\gamma)/2}$. Every time $\alpha$ intersects $\beta$ it has to cross the normal neighborhood from side to side and, hence, it picks up length in the measure of at least $e^{-\ell(\beta)/2}$. Thus $\ell(\alpha)\ge i(\alpha,\beta)e^{-\ell(\beta)/2}$ as claimed.
\end{proof}

\subsection{Subsurface projections}
The geometry of the curve graph is reminiscent of the complex nesting relations between subsurfaces of the ambient surface and exhibits a hierarchical structure as described by Masur and Minsky \cite{MasurMinsky:II}. A fundamental tool to understand it is a surgery procedure known as subsurface projection that we now define.

\begin{dfn}[Non-Annular Subsurface Projection]
Let $Z$ be a connected orientable surface of finite type with curve graph $\mc{C}(Z)$. Consider a proper essential non-annular subsurface $W\subset Z$ which is not a 3-holed sphere. 

For every simple closed curve $\alpha\in\mc{C}(Z)$ intersecting essentially $W$ define $\pi_W(\alpha)$ as follows. First put $\alpha$ in minimal position with respect to $\partial W$ and take the collection of isotopy classes of the arcs $\tau$ in $\alpha\cap W$. Each of these arcs has its endpoints on two boundary components (possibly the same) $\partial_\tau^1W,\partial_\tau^2W$. Every component of the boundary of a small tubular neighborhood of $\partial_\tau^1W\cup\tau\cup\partial_\tau^2W$ is an essential non-peripheral curve in $W$. We define the {\em subsurface projection} $\pi_W(\alpha)$ to be the union of all such components as $\tau$ varies in $\alpha\cap W$. 

It is a basic observation of Masur and Minsky (see \cite[Section 2.3]{MasurMinsky:II} and in particular the proof of \cite[Lemma 2.2]{MasurMinsky:II}) that every two curves in $\pi_W(\alpha)$ intersect at most 2 times and, hence, have distance at most 4 in $\mc{C}(W)$ by virtue of Lemma \ref{lem:intersec dist}. In particular ${\rm diam}_{\mc{C}(W)}(\pi_W(\alpha))\le 4$.

When $\gamma\subset Z$ is a multicurve which has at least one component intersecting $W$ essentially one defines $\pi_W(\gamma)$ to be the union of $\pi_W(\alpha)$ over all components $\alpha\subset\gamma$ that intersect essentially $W$. Again, we have ${\rm diam}_{\mc{C}(W)}(\pi_W(\gamma))\le 4$.
\end{dfn}

\subsection{Laminations and Gromov boundary}
It is a fundamental discovery of Masur and Minsky \cite{MasurMinsky:I} that the curve graph $(\mc{C}(Z),d_{\mc{C}(Z)})$ of a finite type surface $Z$ with $\chi(Z)<0$ and different from a 3-holed sphere is a Gromov hyperbolic space of infinite diameter. By work of Klarreich \cite{Kla}, points on the Gromov boundary $\partial\mc{C}(Z)$ are in bijective correspondence with certain {\em laminations} on $Z$ as we now discuss. 

In general, laminations can be understood as objects arising as Hausdorff limits of sequences of simple closed geodesics on $Z$. More formally, we give the following standard definition.  

\begin{dfn}[Lamination]
A {\em lamination} on a hyperbolic surface $(Z,\sigma)$ of finite area is a closed subset $\lambda\subset Z$ which can be written as a union of pairwise disjoint complete simple geodesics, the {\em leaves} of the lamination. A lamination is called {\em maximal} if each component of the complement $Z-\lambda$ is the interior of an ideal hyperbolic triangle.
\end{dfn}

Among all laminations we will be interested in essentially two extreme cases: When the lamination is an ideal triangulation of $(Z,\sigma)$ and when the lamination is an ending lamination, the kind of lamination that appears as end invariant for a topological end of a hyperbolic 3-manifold diffeomorphic to $Z\times\mb{R}$.

\begin{dfn}[Ending Laminations]
A lamination $\lambda\subset (Z,\sigma)$ is {\em minimal} if every half-leaf is dense in $\lambda$ and it is {\em filling} if it intersects every essential simple closed curve of $Z$. An {\em ending lamination} is a lamination that is both minimal and filling. We denote by $\mc{EL}(Z,\sigma)$ the set of ending laminations.   
\end{dfn}

\begin{thm}[Klarreich \cite{Kla}]
\label{thm:klarreich}
Fix a reference hyperbolic metric $(Z,\sigma)$. There is a bijective map $\partial\mc{C}(Z)\to\mc{EL}(Z,\sigma)$.  Moreover, if $\alpha_n\in\mc{C}(Z)$ is a sequence of curves that converges to a point on $\partial\mc{C}(Z)$ corresponding to a lamination $\lambda\in\mc{EL}(Z,\sigma)$, then every accumulation point of the sequence of geodesic representatives of $\alpha_n$ on $(Z,\sigma)$ in the Hausdorff topology is a lamination $\lambda'$ that contains $\lambda$.   
\end{thm}

\subsection{Subsurface projections for laminations}
Every lamination $\lambda\in\mc{EL}(Z,\sigma)$ has a well defined subsurface projection to every proper essential non-annular subsurface $W\subset Z$ (with totally geodesic boundary) that it intersects essentially. The definition is exactly as the one given for curves: The lamination $\lambda$ intersects $W$ in a collection of properly embedded pairwise disjoint essential arcs. Despite the fact that this collection is typically uncountable, its arcs fall into finitely many isotopy classes (relative to the boundary). For each such arc $\tau\subset \lambda\cap W$ we consider the simple closed curves obtained as boundaries of the regular neighborhood of $\tau$ union the boundary components of $\partial W$ that it connects. The union of all those simple closed curves as $\tau$ varies among $\lambda\cap W$ is the subsurface projection $\pi_W(\lambda)$. As before ${\rm diam}_{\mc{C}(W)}(\pi_W(\lambda))\le 4$.

We can deduce from Theorem \ref{thm:klarreich} that subsurface projections behave well with respect to the convergence to the boundary.

\begin{lem}
\label{lem:projection of limit}   
If $\alpha_n\in\mc{C}(Z)$ converges to $\lambda\in\mc{EL}(Z,\sigma)$ and $\lambda$ intersects essentially a non-annular subsurface $W$, then for every sufficiently large $n$ we have $\pi_W(\lambda)\subset\pi_W(\alpha_n)$.
\end{lem}

\begin{proof}
Represent $W$ as a subsurface of $(Z,\sigma)$ with totally geodesic boundary. Recall that, by Theorem \ref{thm:klarreich}, the convergence $\alpha_n\to\lambda$ implies that (the geodesic representatives of) the curves $\alpha_n$ converge in the Hausdorff topology to a lamination $\lambda'$ of $(Z,\sigma)$ which contains $\lambda$. Pick one representative in $\tau\subset\lambda\cap W$ for every isotopy class of arcs in $\lambda\cap W$. Pick $\ep>0$ so small that the $\ep$-neighborhoods of the arcs are pairwise disjoint embedded strips in $W$. By Hausdorff convergence $\alpha_n\to\lambda'$ and the fact that $\lambda\subset\lambda'$, we have that $\alpha_n$ must cross each of those strips from side to side for every $n$ large enough. By definition, this implies that $\pi_W(\lambda)\subset\pi_W(\alpha_n)$ for every $n$ large enough.   
\end{proof}
    
\section{Geometry and topology of 3-manifolds}
\label{sec:4}

In this section we review some basic 3-dimensional geometry and topology which is necessary to state Thurston's Hyperbolization Theorem (Theorem \ref{thm:hyperbolization}) and some deformation theory of hyperbolic 3-manifolds. 

We discuss in detail the topology of the drilled manifold that interests us, namely, 
\[
M=S\times[-1,1]-V,
\]
where $V$ is a tubular neighborhood of $\gamma_Y\times\{0\}$ and $\gamma_Y\subset S$ is the multicurve obtained by collapsing parallel components of $\partial Y$ to single curves ($Y\subset S$ is a non-annular proper connected subsurface). Crucially, we check that $(M,\partial V)$ satisfies the hyperbolization assumptions and, hence, it admits a hyperbolic structure. 

We then discuss end invariants and the parameterization of all the hyperbolic structures on $M$ to be able to guarantee the existence and uniqueness of one with prescribed end invariants $\nu^-,\nu^+$.

\subsection{3-dimensional topology}
We begin with some standard notions in 3-dimensional topology, namely irreducibility and atoroidality.

\begin{dfn}[Irreducible]
An orientable 3-manifold $M$ is {\em irreducible} if every 2-sphere $\mb{S}^2$ embedded in $M$ bounds an embedded 3-ball $\mb{S}^2=\partial\mb{B}^3$.     
\end{dfn}

\begin{dfn}[Atoroidal]
A compact orientable 3-manifold $M$ is {\em atoroidal} if every $\pi_1$-injective map from a 2-torus or a Klein bottle to $M$ is homotopic into $\partial M$.     
\end{dfn}

As we are interested in compact 3-manifolds admitting hyperbolic structures with cusps, we introduce also the notion of pared 3-manifolds. Heuristically speaking, the definition collects various necessary properties that a compact 3-manifold $M$ with some designated parabolic locus $P\subset \partial M$ must have in order to admit a hyperbolic metric for which every component of $P$ represents a cusp. 

\begin{dfn}[Pared 3-Manifold]
\label{def:pared}
Let $M$ be a compact orientable irreducible 3-manifold with non-empty boundary $\partial M$ and let $P\subset\partial M$ be a subsurface whose components are $\pi_1$-injective annuli and 2-tori. We say that $(M,P)$ is a {\em pared 3-manifold} if the following conditions hold.   
\begin{enumerate}
    \item{Every non-cyclic abelian subgroup of $\pi_1(M)$ is conjugate into the fundamental group of a component of $P$.} 
    \item{Every $\pi_1$-injective map $\phi:(A,\partial A)\to(M,P)$ is homotopic, as a map of pairs, to a map with image in $P$. Here $A=\mb{S}^1\times[0,1]$.}
\end{enumerate}
\end{dfn}

Having talked about cusps, we next expand this topic and discuss the thick-thin decomposition of hyperbolic 3-manifolds as we have done for hyperbolic surfaces.

\subsection{Tubes and cusps}
The thin parts in dimension 3 are still very well behaved and admit a nice description which we now review.

\begin{dfn}[Margulis Tubes and Cusp Neighborhoods]
Let $\ep_0>0$ be a Margulis constant for dimension 3. If $(N,\rho)$ is a hyperbolic 3-manifold, then, by the thick-thin decomposition (see \cite[Chapter 4]{Martelli}), every connected component of the {\em $\ep_0$-thin region}
\[
(N,\rho)_{\ep_0}:=\{x\in N\,|\,{\rm inj}_x(N)<\ep_0\}
\]
can be of three types.  
\begin{itemize}
    \item{A {\em $\ep_0$-Margulis tube} around a simple closed geodesic $\gamma\subset N$ of length $\ell_\rho(\gamma)<2\ep_0$, a component which we denote by $\mb{T}_{\ep_0}(\gamma)$ or $\mb{T}_{\ep_0}(\gamma,N)$ if we need to specify the ambient manifold. Topologically, it is a tubular neighborhood of $\gamma$ (a solid torus diffeomorphic to $\mb{D}^2\times\mb{S}^1$).}
    \item{A {\em $\ep_0$-cuspidal neighborhood} of a rank one cusp of $N$ which we denote by $\mb{C}_{\ep_0}(\gamma)$ or $\mb{C}_{\ep_0}(\gamma,N)$ if we need to specify the ambient manifold. Topologically, it is diffeomorphic to an infinite annulus $\mb{S}^1\times\mb{R}$ times $\mb{R}$ and $\gamma$ denotes the homotopy class of the curve winding once around the annulus $\mb{S}^1\times\mb{R}$.}
    \item{A {\em $\ep_0$-cuspidal neighborhood} of a rank two cusp of $N$ which we denote by $\mb{C}_{\ep_0}(T)$ or $\mb{C}_{\ep_0}(T,N)$ if we need to specify the ambient manifold. Topologically, it is diffeomorphic to a 2-torus $\mb{T}$ times $\mb{R}$ and $T$ denotes the homotopy class of the torus fiber $\mb{T}\times\{0\}$. The metric of $N$ restricts to a flat metric on the boundary of the cusp neighborhood $\partial\mb{C}_{\ep_0}(T)$.}
\end{itemize}  
Again, the complement of the $\ep_0$-thin region is the {\em $\ep_0$-thick part}.
\end{dfn}

\subsection{Hyperbolic 3-manifolds and Kleinian groups}
Next we discuss the relation between hyperbolic manifolds and Kleinian groups with the goal of describing the deformation theory of hyperbolic structures on a fixed manifold and introducing some important objects and definitions that will play a role in later sections. We begin with the following definition.

\begin{dfn}[Kleinian Group]
A Kleinian group is a discrete subgroup of $\Gamma<{\rm Isom}^+(\mb{H}^3)$. In this article we will always assume that $\Gamma$ is torsion-free and not virtually abelian.   
\end{dfn}

If we have a complete orientable hyperbolic 3-manifold $M$, its universal cover is always isometric to $\mb{H}^3$ and a choice of an isometry $\tilde{M}\to\mb{H}^3$ conjugates the deck group action of $\pi_1(M)$ to the one of a discrete torsion free subgroup $\Gamma<{\rm Isom}^+(\mb{H}^3)$. The isomorphism $\pi_1(M)\to\Gamma$ is the so-called {\em holonomy representation}. With the above notations $M$ is isometric to $\mb{H}^3/\Gamma$.

A Kleinian group acts on hyperbolic 3-space but also on the compactification $\mb{H}^3\cup\partial\mb{H}^3$. The action there is no more properly discontinuous and free but it splits into two parts. 

\begin{dfn}[Limit Set and Domain of Discontinuity]
Let $\Gamma<{\rm Isom}^+(\mb{H}^3)$ be a Kleinian group. The {\em limit set} of $\Gamma$ is the set $\Lambda\subset\partial\mb{H}^3$ of accumulation points in $\partial\mb{H}^3$ of an (any) orbit $\Gamma\cdot o$ (where $o\in\mb{H}^3$ is any chosen point). The {\em domain of discontinuity} of $\Gamma$ is the complement of the limit set $\Omega:=\partial\mb{H}^3-\Lambda$. It is a standard fact that the Kleinian group $\Gamma$ acts freely and properly discontinuously on $\mb{H}^3\cup\Omega$. The quotient $M_\Gamma=(\mb{H}^3\cup\Omega)/\Gamma$ is a 3-manifold with boundary $\partial M=\Omega/\Gamma$. 
\end{dfn}

The last Kleinian object that we introduce is the so-called {\em convex core}, which is a canonical convex submanifold of $\mb{H}^3/\Gamma$ associated with the limit set. 

\begin{dfn}[Convex Core]
Let $\Gamma<{\rm Isom}^+(\mb{H}^3)$ be a Kleinian group. The {\em convex core} of the associated 3-manifold $M_\Gamma$ is the quotient ${\rm CC}(M_\Gamma):={\rm CH}(\Lambda)/\Gamma$ where ${\rm CH}(\Lambda)\subset\mb{H}^3$ is the convex hull of the limit set $\Lambda$. It is the smallest closed $\Gamma$-invariant convex subset of $\mb{H}^3$. 
\end{dfn}

\subsection{Hyperbolic structures}
With the above notations we can now define a hyperbolic structure on a compact 3-manifold $M$ with designated parabolic locus $P\subset \partial M$.

\begin{dfn}[Hyperbolic Structure]
Let $(M,P)$ be a pared 3-manifold. A hyperbolic structure on $(M,P)$ is a hyperbolic metric $({\rm int}(M),\rho)$ with associated holonomy representation $\phi:\pi_1(M)\to{\rm Isom}^+(\mb{H}^3)$ such that the holonomy $\phi(\pi_1(P_j))$ is parabolic for every component $P_j\subset P$. Note that we have an orientation preserving isometry $\Phi:{\rm int}(M)\to\mb{H}^3/\phi(\pi_1(M))$.    
\end{dfn}

\begin{dfn}[Geometrically Finite]
A hyperbolic structure on $(M,P)$ with holonomy $\phi$ is geometrically finite if the isometry $\Phi:{\rm int}(M)\to\mb{H}^3/\phi(\pi_1(M))$ extends to a diffeomorphism $M-P\to(\mb{H}^3\cup\Omega)/\phi(\pi_1(M))$.   
\end{dfn}

Thurston's Hyperbolization Theorem states that the obvious topological obstructions to the existence of a hyperbolic metric with specified parabolic locus given in the definition of a pared 3-manifold $(M,P)$ are in fact the only obstructions.

\begin{thm}[{Thurson's Hyperbolization Theorem, see \cite[Theorem 1.43]{Kapovich:hyperbolization}}]
\label{thm:hyperbolization}
Every oriented irreducible atoroidal pared 3-manifold $(M,P)$ with non-empty boundary has a geometrically finite hyperbolic structure. 
\end{thm}

\subsection{Hyperbolization}
We now discuss in details the 3-manifold that interests us.

Consider the interval bundle $S\times[-1,1]$. Let $\gamma_Y\subset S$ be the multicurve obtained by collapsing parallel components of $\partial Y$ to a single curve (where $Y\subset S$ is a proper essential non-annular connected subsurface). Let $V$ be a small tubular neighborhood of $\gamma_Y\times\{0\}$. Set
\[
M:=S\times[-1,1]-V.
\]

We have the following.

\begin{lem}
\label{lem:hyperbolization1}
$(M,\partial V)$ is a pared 3-manifold.    
\end{lem}

\begin{proof}
We prove that $(M,\partial V)$ has the following properties.

{\bf Irreducibility}. As $S\times[-1,1]$ is irreducible (its universal covering is $\mb{R}^2\times[-1,1]$), every embedded 2-sphere $\mb{S}^2\subset M$ bounds a 3-ball $\mb{B}^3$ in $S\times[-1,1]$. Such a 3-ball cannot contain any component of $V$ as each of them is $\pi_1$-injective in $S\times[-1,1]$, thus $\mb{B}^3\subset M$.

{\bf Boundary incompressible}. Consider a component $V_0$ of $V$ with core curve $\gamma_0$ and suppose that $\partial V_0$ is compressible in $M$. Then, by Dehn's Lemma (see \cite{Deh,Pap}), there exists an essential simple closed curve $\alpha\subset\partial V_0$ that bounds a properly embedded disk $D$ in $M$. Note that $\alpha$ must be homotopically trivial also in $V$ as $V$ is $\pi_1$-injective in $S\times[-1,1]$ and $\alpha$ is null-homotopic in $S\times[-1,1]$. Again, by Dehn's Lemma, $\alpha$ must be the meridian of $\gamma_0$, that is $\alpha$ bounds a properly embedded disk $D'\subset V_0$ intersecting $\gamma_0$ once. So we obtained a 2-sphere $D\cup D'\subset S\times[-1,1]$ intersecting $\gamma_0$ exactly once. Such a sphere cannot bound a ball in $S\times[-1,1]$ and this is not possible as $S\times[-1,1]$ is irreducibile.   

{\bf Non-Seifert fibered}. Each boundary component of a Seifert fibered 3-manifold is either a 2-torus or a Klein bottle (the only surfaces that can be foliated by circles). As the boundary of $M$ contains two components that are orientable surfaces of genus at least 2, it cannot be Seifert fibered.

We next discuss {\bf atoroidality}. By the Torus Theorem (which has a long history see \cite{Feu1,Feu2,JS,Joh79,Sco}, see instead \cite[Theorem 1.40]{Kapovich:hyperbolization} for the formulation we use here), an orientable irreducible non-Seifert fibered compact 3-manifold with a $\mb{Z}^2$ subgroup which is not conjugate in the fundamental group of a boundary component contains a $\pi_1$-injective 2-torus or Klein bottle which is not parallel to the boundary. Note that since we are assuming that the manifold is orientable, the boundary of a tubular neighborhood of a $\pi_1$-injective Klein bottle which is not isotopic into the boundary is a $\pi_1$-injective torus which is not isotopic into the boundary. So, in order to prove that Property (1) of Definition \ref{def:pared} holds for $M$ it is enough to show that $M$ is {\em geometrically atoroidal}, that is, it does not contain $\pi_1$-injective embedded tori non isotopic to the boundary. 

{\bf Geometric atoroidality}. As $S\times[-1,1]$ is atoroidal ($\pi_1(S\times[-1,1])$ does not contain any $\mb{Z}^2$ subgroup), every embedded 2-torus $\mb{T}^2\subset M$ bounds a solid torus $\mb{S}^1\times\mb{D}^2\subset S\times[-1,1]$. If $\mb{S}^1\times\mb{D}^2$ does not contain any component of $\gamma_Y$ then it is contained in $M$ and there is nothing to add. If $\mb{S}^1\times\mb{D}^2$ contains at least one component then it contains exactly one component of $\gamma_Y\times\{0\}$ as every component is $\pi_1$-injective and primitive in $\pi_1(S\times[-1,1])=\pi_1(S)$ and different components of $\gamma_Y\times\{0\}$ are not homotopic. Let $\gamma\subset\gamma_Y\times\{0\}$ be a component contained in $\mb{S}^1\times\mb{D}^2$. Up to isotopy, we can assume that small tubular neighborhoods $U$ of $\gamma$ and $\mb{S}^1\times\mb{D}_r$ of $\mb{S}^1\times\{o\}$ are disjoint (where $o$ is the origin in $\mb{D}^2$ and $\mb{D}_r$ is a disk of radius $r$ around it). We now show that the boundary $\partial U$ is parallel to $\mb{T}^2$. In order to do so, we first argue that it is $\pi_1$-injective in $M$. Suppose this is not the case. Then, by Dehn's Lemma, there exists an essential curve $\alpha\subset\partial U$ that bounds a properly embedded disk $D$ in $M-U$. Let $\alpha=pm+ql$ in $\pi_1(\partial U)=m\mb{Z}\oplus l\mb{Z}$ where $m$ is the meridian of $U$ (the homotopy class of the curve of $\partial U$ bounding a disk in $U$) and $l$ is a longitude (it maps to $\gamma$ under the map $\pi_1(\partial U)\to\pi_1(S\times[-1,1])$). The image of $\alpha$ in $\pi_1(S\times[-1,1])$ is $\gamma^q$. Since $\alpha$ bounds a disk it is null-homotopic in $M$, hence $q=0$ and $\alpha=pm$. As $\alpha$ is also simple, we have $p=\pm1$ and $\alpha=\pm m$. Recall that $\pm m$ bounds a properly embedded disk $D_U\subset U$. Gluing the disks $D\subset M-U$ to $D_U\subset U$ along their common boundary $\alpha=D\cap D_U$ we obtain a 2-sphere $S^2\subset S\times[-1,1]$ that meets $\gamma$ in exactly one point, namely $\gamma\cap D_U$. However, the manifold $S\times[-1,1]$ does not contain any such sphere (for example this sphere is not homologically trivial and cannot bound a 3-ball violating the irreducibility). We conclude that $\partial U$ must be $\pi_1$-injective in $M$. Now recall that $\partial U\subset\mb{S}^1\times\mb{D}^2-\mb{S}^1\times\mb{D}_r$ which is diffeomorphic to $\mb{T}^2\times[-1,1]$. By a classical theorem of Waldhausen \cite[Lemma 5.3]{W68}, the only $\pi_1$-injective closed surfaces in $\mb{T}^2\times[-1,1]$ are the ones that are isotopic to level surfaces. This shows that $\mb{T}^2$ is isotopic to $\partial U$.

{\bf Essential annuli}. Lastly, let us check Property (2) of Definition \ref{def:pared}. Let $\phi:(A,\partial A)\to (M,P)$ be a $\pi_1$-injective map where $A=\mb{S}^1\times[0,1]$. Suppose by contradiction that it is not homotopic into $P$ relative to the boundary. Then, by the Annulus Theorem \cite{Feustel77} we can assume that $\phi$ is an embedding. Let $V_0,V_1$ be the components of $V$ containing $\alpha_0=\phi(\partial_0A),\alpha_1=\phi(\partial_1A)$ (possibly the same component). Denote by $\gamma_0,\gamma_1$ the core curves of $V_0,V_1$. As $\pi_1(V_j)=\mb{Z}\gamma_j$, we have that $\alpha_j$ is homotopic to a $k_j$-th power of $\gamma_j$ within $V_j$ for $j=0,1$. Since $\alpha_0,\alpha_1$ are homotopic in $M$, we deduce that $\gamma_0^{k_0},\gamma_1^{k_1}$ are homotopic in $S\times[-1,1]$, and, hence, also in $S$. This implies that $\gamma_0=\gamma_1$ and $V_0=V_1$. Since $\phi$ is an embedding, $\alpha_0,\alpha_1\subset\partial V_0$ are disjoint essential curves on the 2-torus $\partial V_0$. In particular, they divide $\partial V_0$ into two annuli $A',A''$. Cut $M$ along $A$ to obtain a 3-manifold $M'$ with two distinguished boundary components $T'=A\cup A',T''=A\cup A''$ which are 2-tori. None of them can be compressible in $M'$ otherwise $M'$ would be a solid torus. As $A$ is incompressible in $M$ and $M$ is irreducible, it is well-known that being incompressible in $M-A$ is equivalent to being incompressible in $M$ (see for example \cite[Proposition 9.4.9]{Martelli}), in particular both $T',T''\subset M$ are incompressible. By atoroidality of $M$, necessarily $T',T''$ are isotopic $\partial V_0$. By classical results of Waldhausen \cite[Lemma 5.3]{W68}, it follows that $T',T''$ are parallel to $\partial V_0$, but this contradicts the fact that $A$ cannot be homotoped into $\partial V_0$ relative to the boundary.  
\end{proof}

\begin{cor}
\label{cor:geom finite}
$(M,\partial V)$ admits a geometrically finite structure.
\end{cor}

\begin{remark}
While we chose to follow a purely topological approach to prove Corollary \ref{cor:geom finite} (relying on Theorem \ref{thm:hyperbolization}), it is also possible to obtain the same result with various other methods based on the deformation theory of hyperbolic structures on $S\times[-1,1]$. We briefly discuss a couple. 

First, by work of Bonahon and Otal \cite{BoO}, it is possible to equip $S\times(-1,0]-\gamma_Y\times\{0\}$ with a complete hyperbolic structure with totally geodesic boundary and rank one cusps at $\gamma_Y\times\{0\}$. Doubling such structure along $(S-\gamma_Y)\times\{0\}$, we obtain a complete hyperbolic metric on $S\times(-1,1)-\gamma_Y\times\{0\}$ (with little effort, in the construction one can also ensure that the resulting structure is geometrically finite). 

Second, one can obtain a geometrically finite hyperbolic structure on $S\times[-1,1]-\gamma_Y\times\{0\}$ on the boundary of the deformation space of geometrically finite hyperbolic structure $S\times[-1,1]$. For example, Brock \cite{Bro} shows that starting with any point $X$ in the Teichmüller space $\T(S)$ of $S$ the sequence of quasi-Fuchsian manifolds $Q(X,D_{\gamma_Y}^nX)$ with end invariants given by $X,D_{\gamma_Y}^nX\in\T(S)$ where $D_{\gamma_Y}$ is the (right) Dehn twist around $\gamma_Y$ (see Section \ref{subsec:end inv} below for a brief description of the end invariants) converges geometrically to a geometrically finite hyperbolic structure on $S\times[-1,1]-\gamma_Y\times\{0\}$.
\end{remark}

\subsection{End invariants and parameterization}
\label{subsec:end inv}
Once we know that $(M,\partial V)$ has a geometrically finite structure, the combination of various fundamental results in the deformation theory of hyperbolic 3-manifolds (most notably the solution of the Ending Lamination Conjecture \cite{M10,BrockCanaryMinsky:ELC2} and realization results such as \cite{Oh}, \cite[Theorem 1.3]{NS12}) allows us to parameterize all hyperbolic structures on $(M,\partial V)$ up to isometries homotopic to the identity in terms of the so-called {\em end invariants} which we now describe in the restricted setting of our interest. For a more in-depth discussion we refer to \cite[Section 2]{M10}.

\begin{dfn}[End Invariants]
Let $(\mb{M},\mb{P})$ be either $(S\times[-1,1],\emptyset)$ or $(S\times[-1,1]-V,\partial V)$. Consider a hyperbolic structure on the pared 3-manifold $(\mb{M},\mb{P})$. 

Classical works of Scott \cite{Sco}, McCullough \cite{CMcC}, Kulkarni and Shalen \cite{KS89} on the (relative) compact core, combined with analysis of the structure of the ends carried out by Bonahon \cite{Bo86} and Canary \cite{Can93} (building on Thurston's works \cite{ThuNotes}) provide us the following picture.

There are finitely many rank one cusps in $\mb{M}$ and each of them corresponds to a distinguished essential non-peripheral simple closed curve $\gamma\subset\partial\mb{M}-\mb{P}$. The component of the $\ep_0$-thin part corresponding to such a cusp is denoted by $\mb{C}_{\ep_0}(\gamma)$, its fundamental group $\pi_1(\mb{C}_{\ep_0}(\gamma))$ is conjugate to the cyclic subgroup of $\pi_1(\mb{M})$ generated by $\gamma$. 

Simple closed curves $\gamma,\gamma'\subset \partial\mb{M}-\mb{P}$ corresponding to different cusp components $\mb{C}_{\ep_0}(\gamma),\mb{C}_{\ep_0}(\gamma')$ are disjoint. Thus, the rank one cusps define two multicurves $P^-$ and $P^+$, one on $S\times\{-1\}$ and one on $S\times\{1\}$ (the so-called {\em accidental parabolics)}. Denote by $A^-,A^+$ some tubular neighborhoods of such multicurves. 

For every complementary component $W\subset S-A^-,S-A^+$, the hyperbolic structure on $(\mb{M},\mb{P})$ determines either a unique finite area hyperbolic metric on it $(W,\sigma_W)$ or a unique minimal filling lamination $\lambda_W\in\mc{EL}(W)$. 

The collection of $A^-$ (resp. $A^+$) and the parameters for $W$ when the subsurface varies among the components of $S-A^-$ (resp. $S-A^+$) is the so-called {\em negative} (resp. {\em positive}) {\em end invariant} $\nu^-$ (resp. $\nu^+$). 
\end{dfn}

We now state a fundamental relation between end invariants and the geometry of $M$. The first part of the following theorem is due to Ahlfors \cite{Ahl} and the second is due to Thurston \cite{ThuNotes} and Bonahon \cite{Bo86}.  

\begin{thm}
\label{thm:end to length}
Consider a hyperbolic structure on the pared 3-manifold $(M,\partial V)$. Let $P^\pm\subset S$ be the accidental parabolics of the structure. Consider a component $W$ of $S-P^\pm$. Then the following holds.
\begin{itemize}
    \item{If the invariant associated to $W$ is a finite area hyperbolic metric $(W,\sigma_W)$ then we have $\ell_M(\alpha)\le 2\ell_{\sigma_W}(\alpha)$ for every essential curve $\alpha\subset W$.}
    \item{If the invariant associated to $W$ is a minimal filling lamination $\lambda_W\in\mc{EL}(W)$ then there exists a sequence of simple closed curves $\alpha_n\in\mc{C}(W)$ converging to $\lambda_W$ in the boundary of the curve graph and such that $\ell_M(\alpha_n)\le 2\cdot{\rm arccosh}(|\chi(W)|+1)$.}
\end{itemize}
\end{thm}

Combining the solution of the Ending Lamination Conjecture \cite{BrockCanaryMinsky:ELC2} and realization of end invariants results \cite{Oh}, \cite[Theorem 1.3]{NS12} one can obtain a parameterization theorem which establishes a bijective correspondence between hyperbolic structures on $(\mb{M},\mb{P})$ and suitable collections of pairs of end invariants $\nu^-,\nu^+$. For us, it will be sufficient the following corollary.

\begin{thm}
\label{thm:hyperbolization2}
If $S\times[-1,1]$ admits a hyperbolic structure with end invariants $\nu^-,\nu^+$, then $(S\times[-1,1]-V,\partial V)$ admits a unique complete hyperbolic structure realizing the same end invariants $\nu^-,\nu^+$.     
\end{thm}

\section{Effective Efficiency}
\label{sec:effeff}

In this section we prove the effective efficiency result which is one of the main ingredients of the proof of Theorem \ref{thm:main1}. The statement is the following. 

\begin{thm}
\label{thm:effeff}
Let $\ep_0<{\rm arcsinh}(1/4)$ be a Margulis constant for dimensions 2 and 3. Let $S$ be a connected closed orientable surface of genus at least 2. Consider a proper essential connected non-annular subsurface $Y\subset S$ and denote by $\gamma_Y\subset S$ the multicurve obtained by collapsing parallel components of $\partial Y$ to a single essential simple closed curve. Let $Q$ be a hyperbolic manifold diffeomorphic to $S\times(-1,1)$ with rank one cusps at the multicurve $\gamma_Y$. Suppose that $\alpha\subset S$ is an essential simple closed curve that intersects $\gamma_Y$ essentially and whose geodesic representative in $Q$ has length $L\ge 0$ (with the convention that $\alpha$ is parabolic when $L=0$). Then the subsurface projection of $\alpha$ to $Y$ can be represented by a simple closed curve on $Y$ whose geodesic representative in $Q$ has length at most 
\[
4\ep_0+\frac{16\pi}{\sinh(\ep_0^{10}/2^{43}\pi^{8}|\chi(S)|^{16})}+\frac{256T(\pi/6)\pi^2|\chi(S)|^2}{{\rm vol}(\ep_0^{10}/2^{43}\pi^{8}|\chi(S)|^{16})^2}+2L
\]
where ${\rm vol}(\bullet)$ denotes the volume of balls in $T^1\mb{H}^2$.
\end{thm}

\begin{remark}
\label{rmk:length coeff}
Here $T(\bullet)$ is the constant of Lemma \ref{lem:quasi-geodesic}. We remark that, using $T(\pi/6)=2\cdot 10^5$, $\sinh(x)\ge x$ and ${\rm vol}(x)\ge{\rm vol}_{\mb{R}^3}(x)/2=(2\pi/3)x^3$ for small enough $x$ (only depending on $T^1\mb{H}^2$ by the asymptotic expansion of ${\rm vol}(x)$), we can estimate the above formula (in terms of the Margulis constant $\ep_0$) with 
\[
2L+\frac{2^{384}}{\ep_0^{60}}|\chi(S)|^{98}.
\]    
\end{remark}

Before going on with the proof we need to recall some facts about {\em pleated surfaces}, which are tools introduced by Thurston to study the geometry of ends of hyperbolic 3-manifolds, and review a couple of their structural properties.

\subsection{Pleated surfaces}
We begin with the following definition.

\begin{dfn}[Pleated Surface]
A {\em pleated surface} consists of the following data: A (possibly disconnected) surface of finite type $Z$ together with a complete finite area hyperbolic metric $(Z,\sigma)$ and a proper map $f:Z\to M$ in a hyperbolic 3-manifold $M$ such that:
\begin{itemize}
  \item{The map $f$ is path isometric, meaning that $f$ sends a rectifiable path to a rectifiable path of the same length.}
  \item{There is a $\sigma$-geodesic lamination $\lambda$ such that $f$ maps each leaf $\ell\subset\lambda$ to a complete geodesic in $M$ (possibly closed if the leaf $\ell$ is closed) and $f$ is a totally geodesic immersion on $Z-\lambda$.}
\end{itemize}  

We will also say that $f$ {\em maps geodesically} the lamination $\lambda$.
\end{dfn}

Pleated surfaces are very useful as they are abundant and can go almost everywhere in a hyperbolic 3-manifold diffeomorphic to $Z\times\mb{R}$. We have the following standard realization result.

\begin{pro}[{see \cite[Theorem I.5.3.6]{CEG:notes_on_notes}}]
\label{pro:realization}
Let $g:Z\to M$ be a $\pi_1$-injective map of a finite type surface $Z$ without boundary to a hyperbolic manifold $M$ mapping properly each end of $Z$ into the cuspidal part of $M$. For every ideal triangulation $\lambda\subset Z$ there exists a pleated surface $f:(Z,\sigma)\to M$ properly homotopic to $g$ and mapping geodesically $\lambda$.   
\end{pro}

A very important geometric property of pleated surfaces $f:(Z,\sigma)\to M$ is that they relate nicely the thick-thin decomposition of $(Z,\sigma)$ to that of $M$ as we now describe. The following Lemmas \ref{lem:straight in cusp} and \ref{lem:thick to thick} are all observations due to Thurston, we include proofs to make the constants involved in his arguments effective.

\begin{lem}[{see \cite[Proposition 8.8.4]{ThuNotes}}]
\label{lem:straight in cusp}
Let $\ep_0<{\rm arcsinh}(1/4)$ be a Margulis constant for surfaces. We have the following. Let $(Z,\sigma)$ be a finite-volume hyperbolic surface and $\lambda$ a lamination on $Z$. Then any component of the intersection of $\lambda$ with the $\ep_0$-cuspidal part of $Z$ is an infinite ray orthogonal to the boundary of the $\ep_0$-cuspidal part.
\end{lem}

\begin{proof}
Fix a cusp of $(Z,\sigma)$. Identify the universal cover $\tilde{Z}$ of $(Z,\sigma)$ with $\mb{H}^2$ and lift $\lambda$ to a lamination $\tilde{\lambda}$ of $\mb{H}^2$ invariant under $\Gamma<{\rm Isom}^+(\mb{H}^2)$, the image of the holonomy representation $\pi_1(Z)\to{\rm Isom}^+(\mb{H}^2)$. We work in the upper half plane model $\mb{H}^2=\{z\in\mb{C}\,|\,{\rm Im}(z)>0\}$. Each of the leaves of $\tilde{\lambda}$ is either a vertical ray of a half-circle orthogonal to $\mb{R}$. Up to conjugation, we can assume that the holonomy of the simple closed curve $\gamma$ surrounding a chosen cusp is the parabolic isometry $\phi\in\Gamma$ given by $\phi(z)=z+1$. The observation is that the Euclidean radii of the half-circle components $\ell$ of $\tilde{\lambda}$ is at most $1/2$ otherwise $\ell\cap\phi(\ell)$ would not be empty (the two leaves $\ell,\phi(\ell)$ of $\tilde{\lambda}$ cross). Now, the standard $\ep_0$-collar of the cusp is $\{z\in\mb{C}\,|\,d_\mb{H}^2(z,\phi(z))<2\ep_0\}/\langle\phi\rangle$. By basic hyperbolic geometry, 
\[
d_{\mb{H}^2}(z,\phi(z))=2\cdot{\rm arcsinh}(|z-\phi(z)|/2\sqrt{{\rm Im}(z){\rm Im}(\phi(z))})=2\cdot{\rm arcsinh}({\rm Im}(z)/2)
\]
hence $\{z\in\mb{C}\,|\,d_\mb{H}^2(z,\phi(z))<2\ep_0\}=\{z\in\mb{C}\,|\,{\rm Im}(z)>2\cdot\sinh(\ep_0)\}$. Assuming that $\ep_0<{\rm arcsinh}(1/4)$ the above is contained in the horoball $\mc{O}=\{z\in\mb{C}\,|\,{\rm Im}(z)>1/2\}$. A leaf of $\tilde{\lambda}$ that intersects $\mc{O}$ is a vertical line (which is orthogonal to the boundary horocycle and goes straight towards the center of the horoball). 
\end{proof}

\begin{lem}[{see \cite[Lemma 3.1]{M00}}]
\label{lem:thick to thick}
Let $\ep_0<{\rm arcsinh}(1/4)$ be a Margulis constant for dimensions 2 and 3. Consider a $\pi_1$-injective pleated surface $f:(Z,\sigma)\to Q$ in a hyperbolic 3-manifold $Q$. Choose
\[
\ep_1\le\frac{\ep_0^{10}}{2^{37}\pi^{8}|\chi(Z)|^{16}}.
\]
For every $x\in Z$ we have the following. If $f(x)\in Q_{\ep_1}$ then $x\in(Z,\sigma)_{\ep_0}$.
\end{lem}

\begin{proof}
By Lemma \ref{lem:short gen}, if $x$ is in the $\ep_0$-thick part, then there are two loops $\gamma,\gamma'$ based at $x$ of length at most $L_0=2\log(256\pi^2|\chi(Z)|^4/\ep_0^2)$ generating a free subgroup of $\pi_1(Z,x)$ of rank 2. The image of such loops under $f$ is contained in the $L_0$-neighborhood of $f(x)\in\mb{T}_{\ep_1}$ (as $f$ is 1-Lipschitz). By work of Futer, Purcell, and Schleimer \cite{FPS19} (making effective a theorem of Brooks and Matelsky \cite{BM82}), the distance between the boundaries $\partial\mb{T}_{\ep_1}$ and $\partial\mb{T}_{\ep_0}$ of two standard Margulis tubes is at least ${\rm arccosh}(\ep_0/\sqrt{7.256\ep_1})-0.042$. A similar separation result holds in rank one and rank two cuspidal neighborhoods as well. The computation there is completely elementary: The distance between $\partial\mb{C}_{\ep_0}(\gamma)$ (resp. $\partial\mb{C}_{\ep_0}(T)$) and $\partial\mb{C}_{\ep_1}(\gamma)$ (resp. $\partial\mb{C}_{\ep_1}(T)$) is $\log(\sinh(\ep_0)/\sinh(\ep_1))$ (note that this quantity is larger than ${\rm arccosh}(\ep_0/\sqrt{7.256\ep_1})-0.042$). If $\ep_1\le\ep_0^{10}/2^{37}\pi^{8}|\chi(Z)|^{16}$ then, using $\exp(-0.042)\in(1/2,1)$, we have 
\[
{\rm arccosh}(\ep_0/\sqrt{7.256\ep_1})-0.042\ge{\rm arccosh}(2^{17}\pi^4|\chi(Z)|^8/\ep_0^4)+\log(1/2).
\]
Note that $\log(x)\le{\rm arccosh}(x)$ so
\[
{\rm arccosh}(2^{17}\pi^4|\chi(Z)|^8/\ep_0^4)+\log(1/2)\ge\log(2^{16}\pi^4|\chi(Z)|^8/\ep_0^4)=L_0.
\]

Therefore, if $\ep_1\le\ep_0^{10}/2^{37}\pi^{8}|\chi(Z)|^{16}$ then we have $d_Q(\partial\mb{T}_{\ep_1},\partial\mb{T}_{\ep_0})\ge L_0$ (resp. $d_Q(\partial\mb{C}_{\ep_1},\partial\mb{C}_{\ep_0})\ge L_0$ if $f(x)$ is in a cuspidal neighborhood). Thus the image of the loops $\gamma,\gamma'$ is contained in $\mb{T}_{\ep_0}$ (resp. $\mb{C}_{\ep_0}$), but this is not possible as $Z$ is $\pi_1$-injective and $\langle\gamma,\gamma'\rangle$ is a free group of rank 2 while $\pi_1(\mb{T}_{\ep_0})$ (resp. $\pi_1(\mb{C}_{\ep_0})$) is abelian. We conclude that necessarily ${\rm inj}_x(Z)<\ep_0$.
\end{proof}

\subsection{Proof of Theorem \ref{thm:effeff}}

We can now launch the proof of Theorem \ref{thm:effeff}. 

We subdivide the proof into essential steps taking care of one problem at a time. 

\subsection{Topological preliminaries}
\label{subsec:identifications}
Let us begin with some topological preliminaries.

\subsubsection{Identification of $Q$ with $S\times(-1,1)$}
We make the following choices (allowed by Bonahon's Tameness Theorem \cite{Bo86}). If $\alpha$ is not parabolic, we choose a diffeomorphism of $Q$ (inducing the identity at the level of fundamental groups) with the interior of 
\[
S\times[-1,1]-\gamma_Y\times\{-1\}
\]
so that the cuspidal neighborhood $\mb{C}_{\ep_0}(\gamma)$ for $\gamma\subset\gamma_Y$ is the intersection with $S\times(-1,1)$ of a collar neighborhood $A_\gamma\times\{-1\}$ where $A_\gamma\subset S$ is a tubular neighborhood of $\gamma$.

If $\alpha$ is parabolic, we choose a diffeomorphism of $Q$ (inducing the identity at the level of fundamental groups) with the interior of
\[
S\times[-1,1]-(\gamma_Y\times\{-1\}\cup\alpha\times\{1\})
\]
so that the cuspidal neighborhood $\mb{C}_{\ep_0}(\gamma)$ (resp. $\mb{C}_{\ep_0}(\alpha)$) for $\gamma\subset\gamma_Y$ (resp. $\alpha$) is the intersection with $S\times(-1,1)$ of a collar neighborhood $A_\gamma\times\{-1\}$ (resp. $A_\alpha\times\{1\}$) where $A_\gamma\subset S$ (resp. $A_\alpha\subset S$) is a tubular neighborhood of $\gamma$ (resp. $\alpha$).

We denote by $\iota:S-\gamma_Y\to\bar{Q}$ the inclusion of $(S-\gamma_Y)\times\{-1\}$.

\subsubsection{The curve $\alpha$ as an ideal concatenation}
We put $\alpha$ in minimal position with respect to $\gamma_Y$ and consider the segments of $\alpha-\gamma_Y=\alpha_1\cup\cdots\cup\alpha_n$ where $n=i(\alpha,\gamma_Y)$ (see Figure \ref{fig:minimalposition}). The numbering of the segments is such that $\alpha_{j+1}$ follows $\alpha_j$ along $\alpha$ (indices modulo $n$). Collapsing parallel components of $\alpha-\gamma_Y$ to single ones, we obtain a finite leaved lamination $\lambda$ of $S-\gamma_Y$ (see Figure \ref{fig:lamination}) which we can complete to an ideal triangulation $\lambda'$ of $S-\gamma_Y$ by adding finitely many leaves. 

\begin{figure}[h]
\begin{overpic}[scale=2]{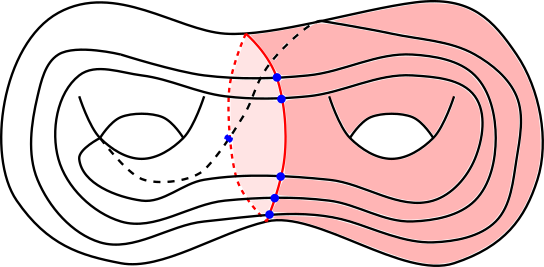}
\put (48,48) {\color{red}$\gamma_Y$}   
\put (90,48) {\color{red}$Y$}   
\end{overpic}
\caption{The curve $\alpha$ as an ideal concatenation.}
\label{fig:minimalposition}
\end{figure}

\begin{figure}[h]
\begin{overpic}[scale=2]{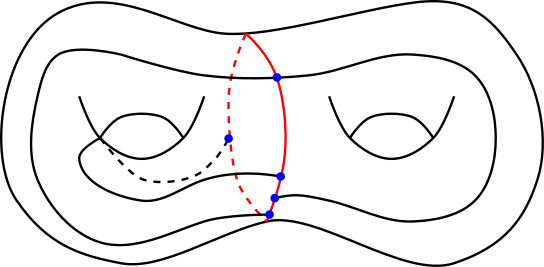}
\end{overpic}
\caption{The finite leaved lamination $\lambda\subset S-\gamma_Y$ obtained from $\alpha-\gamma_Y$ by collapsing parallel components to single ones.}
\label{fig:lamination}
\end{figure}

By Proposition \ref{pro:realization}, there is a pleated surface $f:(S-\gamma_Y,\sigma)\to Q$ properly homotopic to the inclusion $\iota:S-\gamma_Y\to\bar{Q}$ mapping geodesically $\lambda'$. 

The moderate-length surgery promised by the theorem will correspond to a component of $\lambda\cap (S-\gamma_Y,\sigma)_0=\lambda\cap(S-\bigcup_{\gamma\in\gamma_Y}{{\rm\bf cusp}(\gamma,\ep_0)})$ (suitably concatenated with geodesic loops representing the boundary components of the thin part). In order to get some control on the length of such components we use the geometry of an ideal pleated annulus, joining the ideal concatenation $f(\ell_1)\star\ldots\star f(\ell_n)$ to the geodesic representative of $\alpha$ in $Q$. In the end the necessary bound will be deduced from a length-area estimate on such ideal annulus. 

\subsection{Ideal pleated annuli}
\label{subsec:annuli}
Recall that in the previous section we constructed a pleated surface $f:(S-\gamma_Y,\sigma)\to Q$ properly homotopic to the inclusion $\iota:S-\gamma_Y\to\bar{Q}$ and mapping geodesically (an ideal triangulation containing) the lamination $\lambda$ obtained from $\alpha-\gamma_Y$ by collapsing parallel components to single ones (each $\alpha_j$ corresponds to a leaf $\ell_j$ of $\lambda$).

We now describe how to construct a geometrically meaningful homotopy between the ideal concatenation
\[
f(\ell_1)\star\ldots\star f(\ell_n)
\]
and the geodesic representative of $\alpha$ in $Q$, which we denote by $\alpha^*$, or to a cusp when $\alpha$ is parabolic. Various versions of this construction have already been used in the literature (most notably in the proofs of \cite[Theorem 3.3]{Thu2}, \cite[Proposition 5.4]{Ca}, \cite[Theorem 3.5]{M00}). Informally speaking, we will produce a piecewise totally geodesic immersed hyperbolic crown (see Figure \ref{fig:idealannulus})
\[
h:A\to Q
\]
where $A$ is homeomorphic to $\alpha\times[-1,0]-\{x_1,\cdots,x_n\}$ (where $x_j=\alpha_{j-1}\cap\alpha_j$) and is equipped with a hyperbolic metric with totally geodesic boundary $(A,\sigma_A)$. The map $h$ will be a path isometry and will satisfy $h(\alpha_j\times\{-1\})=f(\ell_j)$ and $f(\alpha\times\{0\})=\alpha^*$. Most of the work of this section is to analyze how the geometries of $(A,\sigma_A)$ and of $(S-\gamma_Y,\sigma)$ interact.

We briefly sketch the construction of $h:A\to Q$ (there are two cases depending on whether $\alpha$ is parabolic or not).

\subsubsection{Case $\alpha$ not parabolic}
First consider a proper homotopy $h_0:\alpha\times[-1,0]\to S\times[-1,1]$ between $\alpha\subset S\times\{-1\}$ and the closed geodesic $\alpha^*$.  

Let $\alpha_j$ be the $j$-th segment of $\alpha-\gamma_Y$ (starting at $x_j$ and ending at $x_{j+1}$). As a first step, we properly homotope $\iota(\alpha_j)$ to $f(\ell_j)$. By the homotopy extension property, we extend this homotopy to the whole annulus (relative to $\alpha\times\{0\}$). Call the resulting map $h_1:\alpha\times[0,1]-\{x_1,\cdots,x_n\}\times\{-1\}\to Q$.

Then, we spin each of the arcs $h_1:\{x_j\}\times(-1,0]\to Q$ around $\alpha^*$. In other words, we homotope $h_1:\bigcup_{j\le n}{\{x_j\}\times(-1,0)}\to Q$ to a parameterization of the bi-infinite geodesics $g_j$ spiraling clockwise around $\alpha^*$. By the homotopy extension property, we extend this homotopy to the whole annulus relative to $\alpha\times\{-1\}$. Call the resulting map $h_2:\alpha\times[-1,0)-\{x_1,\cdots,x_n\}\times\{-1\}\to Q$. 

Next, we homotope $h_2:\bigcup_{j\le n}{\alpha_j\times(-1,0)}\to Q$ (relative to the boundary) to a parameterization of the immersed totally geodesic ideal triangle bounded by the bi-infinite geodesics $g_j,\ell_j,g_{j+1}$ obtaining a map $h_3:\alpha\times[-1,0)-\{x_1,\cdots,x_n\}\times\{-1\}\to Q$. Lastly, we extend $h_3$ over the boundary $\alpha\times\{0\}$, the extension $h$ is a parameterization of the geodesic $\alpha^*$.  

The result is an ideal pleated annulus $h:A\to Q$ where $A=\alpha\times[-1,0]-\{x_1,\cdots,x_n\}\times\{-1\}$ and the restriction of $h$ to the $-1,0$-levels coincide with $f$ (formally, with $f\circ\phi_j$ on $\alpha_j$, where $\phi_j:\alpha_j\to\ell_j$ is an orientation preserving homeomorphism) and a parameterization of the geodesic $\alpha^*$ respectively. We denote by $\sigma_A$ the induced path metric on $A$. Endowed with such metric $(A,\sigma_A)$ is a complete hyperbolic surface with totally geodesic boundary and finite area ${\rm Area}(A,\sigma_A)=n\pi$. The boundary component $\alpha\times\{0\}$ has length equal to $\ell(\alpha^*)=L$. 

\begin{figure}[h]
\begin{overpic}[scale=2]{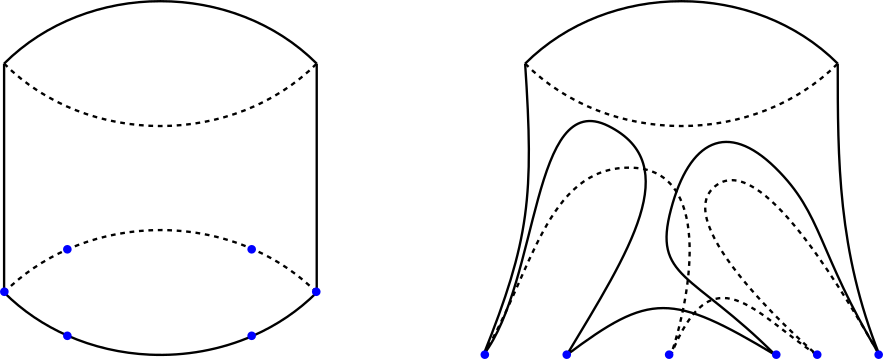}
\put(5,0) {\color{blue}$x_j$}
\put(10,20) {$\alpha\times[-1,0]$}
\put(72,32) {$(A,\sigma_A)$}
\end{overpic}
\caption{Ideal pleated annulus. Non-parabolic case.}
\label{fig:idealannulus}
\end{figure}

\subsubsection{Case $\alpha$ parabolic}
Consider the annulus $A=\alpha\times[-1,1)-\{x_1,\cdots,x_n\}\times\{-1\}\subset Q$. Start with the inclusion $h_0:A\to\bar{Q}$. As before, let $\alpha_j=(x_j,x_{j+1})$ be the $j$-th segment of $\alpha-\gamma_Y$. The inclusion $\iota:\alpha_j\to\bar{Q}$ is properly homotopic to $f:\ell_j\to Q$ (formally, to $f\circ\phi_j$ where $\phi_j:\alpha_j\to\ell_j$ is an orientation preserving homeomorphism). By the homotopy extension property, we extend this homotopy to the whole annulus. Call the resulting map $h_1:\alpha\times[-1,1)-\{x_1,\cdots,x_n\}\times\{-1\}\to Q$.

The proper arcs $h_1:\bigcup_{j\le n}{\{x_j\}\times(-1,1)}\to Q$ are properly homotopic to parameterizations of the bi-infinite geodesic $g_j$ joining the cusp corresponding to $x_j$ to the cusp of $\alpha$. By the homotopy extension property, we extend this homotopy to the whole annulus relative to $\alpha\times\{-1\}$. Call the resulting map $h_2:\alpha\times[-1,1)-\{x_1,\cdots,x_n\}\times\{-1\}\to Q$. Next, we homotope $h_2:\bigcup_{j\le n}{\alpha_j\times(-1,1)}\to Q$ (relative to the boundary) to a parameterization of the ideal triangle bounded by the bi-infinite geodesics $g_j,\ell_j,g_{j+1}$ obtaining a map $h_3:\alpha\times[-1,0)-\{x_1,\cdots,x_n\}\times\{-1\}\to Q$.

The result is an ideal pleated annulus $h:A\to Q$ where $A=\alpha\times[-1,1)-\{x_1,\cdots,x_n\}\times\{-1\}$. The restriction of $h$ to the $-1$-level coincides with $f$ (formally, with $f\circ\phi_j$ on $\alpha_j$, where $\phi_j:\alpha_j\to\ell_j$ is an orientation preserving homeomorphism). Again, we denote by $\sigma_A$ the induced path metric on $A$. Endowed with such metric $(A,\sigma_A)$ is a complete hyperbolic surface with totally geodesic boundary and a cusp corresponding to the end $\alpha\times\{1\}$. As before, the area is ${\rm Area}(A,\sigma_A)=n\pi$.

\subsection{Non-cuspidal segments}
We now select on the ideal concatenation $\ell_1\star\ldots\star \ell_n$ a finite segment $\tau_j\subset\ell_j$ for each component. Those segments lying in $Y$ are the candidates moderate-length surgeries as described at the beginning. In order to select $\tau_j$ we use the thick-thin decomposition of $(S-\gamma_Y,\sigma)$ (the hyperbolic metric induced by the pleated surface $f$).  

Recall that, by our choice of Margulis constant $\ep_0$, a lamination on a finite area hyperbolic surface can intersect the standard $\ep_0$-collar around a cusp only in a ray orthogonal to the boundary of the cusp and going straight towards the end. If $\ell$ is a leaf of $\lambda$, we denote by $\tau=\ell\cap(S-\gamma_Y,\sigma)_0$ where $(S-\gamma_Y,\sigma)_0=(S-\gamma_Y)-\bigcup_{\gamma\subset\gamma_Y}{{\rm\bf cusp}(\gamma,\ep_0)}$. Observe that $\tau$ is a connected subsegment and that $\ell$ is equal to $\tau$ plus two rays entirely contained in the $\ep_0$-cusp neighborhoods of the ends of $S-\gamma_Y$. 

\subsection{No shortcuts}
We make the following purely topological no-shortcuts observation about the segments that we selected. This plays a crucial role in the analysis of the geometry of $A$ analog to the one played by Thurston's Uniform Injectivity (see \cite[Theorem 5.7]{Thu86} or Theorem \cite[Theorem 3.2]{M00}) in the proof of Efficiency of Pleated Surfaces (see \cite[Theorem 3.3]{Thu2} or \cite[Theorem 3.5]{M00}). In fact it can be seen as a topological analog of uniform injectivity. 

\begin{lem}
\label{lem:no homotopy rel endpoints}
Let $f:(S-\gamma_Y,\sigma)\to Q$ and $h:(A,\sigma_A)\to Q$ be as above (see \ref{subsec:annuli}). If $\ell_Y,\ell\subset\partial A$ are connected in $A$ by a properly embedded arc $\xi_A\subset A$ cutting from $A$ a disk $D$ with some points on the boundary removed (corresponding to the ideal points of $A$), then it is not possible to homotope $h(\xi_A)$ relative to the endpoints to an arc $f(\xi_Y)$ where $\xi_Y\subset S-\gamma_Y$ joins the same endpoints of $\xi_A$ on $S-\gamma_Y$.
\end{lem}

Before going to the proof, let us explain the heuristic idea. We proceed by contradiction and assume that it is possible to homotope $h(\xi_A)$ to $f(\xi_Y)$ relative to the endpoints. Under this assumption, we will show that the disk $D\subset A$ can be used to homotope $\alpha\subset S$ to a closed loop intersecting $\gamma_Y$ strictly less than $n$ times (essentially, we can remove via a homotopy the intersections $\alpha\cap\gamma_Y$ corresponding to the ideal vertices of $A$ contained in $D$). This is impossible since $i(\gamma_Y,\alpha)=n$. 

To make the above idea more precise, we will introduce several auxiliary arcs in $A$ and $S-\gamma_Y$ and tracks of the homotopy between $\iota$ and $f$. This makes the proof a little bit technical. To help the reader, we include Figure \ref{fig:homotopy} that provides a schematic picture. We also summarize here the topological properties of $f,h$ that we want to use.
\begin{itemize}
    \item{$\bar{Q}=S\times[-1,1]-\gamma_Y\times\{-1\}$ (see \ref{subsec:identifications}).}
    \item{$f$ is a pleated surface properly homotopic to the inclusion $\iota:S-\gamma_Y\to\bar{Q}$ of the boundary $(S-\gamma_Y)\times\{-1\}\subset\partial\bar{Q}$ and mapping geodesically an ideal triangulation formed by a maximal extension of the lamination $\lambda$ obtained from $\alpha-\gamma_Y=\alpha_1\cup\cdots\cup\alpha_n$ by collapsing parallel components to single ones (see Figures \ref{fig:minimalposition} and \ref{fig:lamination}).}
    \item{$A=\alpha\times[-1,0]-\{x_1,\cdots,x_n\}$ where $x_j=\alpha_{j-1}\cap\alpha_j$ (see Figure \ref{fig:idealannulus}).}
    \item{$h:A\to Q$ is a proper map. The restriction of $h$ to $\alpha\times\{0\}$ is a parameterization of the geodesic representative $\alpha^*$ of $\alpha$ in $Q$. The restriction of $h$ to $\alpha_j\times\{-1\}$ coincides with $f\circ\phi_j$ (where $\phi_j:\alpha_j\to\ell_j$ is an orientation preserving homeomorphism between $\alpha_j$ and the corresponding leaf $\ell_j$ of $\lambda$).}
\end{itemize}

\begin{figure}[h]
\begin{overpic}[scale=2]{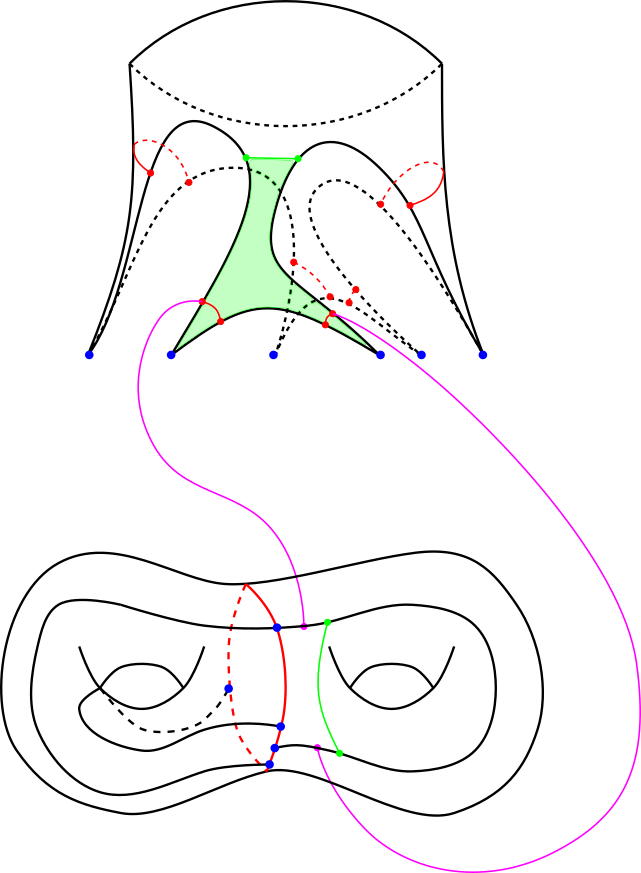}
\put (37,19) {\color{green}$\xi_Y$}   
\put (29,83) {\color{green}$\xi_A$}
\put (17,57) {\color{blue}$x_1$}
\put (19,67) {\color{red}$a_n$}
\put (26,60) {\color{red}$b_n$}
\put (25.5,66) {\color{red}$\eta_n$}
\put (12,48) {\color{purple}$\theta_n^-$}
\end{overpic}
\caption{Schematic picture Lemma \ref{lem:no homotopy rel endpoints}. Above is represented the ideal pleated annulus $h:A\to Q$ and below the pleated surface $f:S-\gamma_Y\to Q$. In red the arcs $\eta_j$ on $A$ with endpoints $a_j,b_j$. Each $\beta_j\subset\alpha_j$ is the subsegment joining $b_{j-1}$ and $a_j$. In purple the tracks of two $a_j,b_j$'s (contained in the cuspidal part of $Q$) under the restriction of the proper homotopy between $f:S-\gamma_Y\to Q$ and $\iota:S-\gamma_Y\to\bar{Q}$. In green the arcs $\xi_A\subset A$ and $\xi_Y\subset S-\gamma_Y$. The former bounds the shaded green disk $D$.}
\label{fig:homotopy}
\end{figure}

\begin{proof}
We argue by contradiction. Suppose that it is possible to homotope $h(\xi_A)$ to $f(\xi_Y)$ relative to the endpoints. The proof has three steps.

The first step of the proof is to replace the ideal concatenation $\ell_1\star\ldots\star\ell_n$ with a suitable finite concatenation by inserting small arcs joining $\ell_j$ to $\ell_{j+1}$ near the ideal point where they meet. Formally, we proceed as follows.

Up to cyclically renumbering the ideal points $x_1,\cdots,x_n$ of $A$, we can assume that the ones contained in the disk $D$ are exactly $x_1,\cdots,x_s$. By properness of $h:A\to Q$ for every $x_j$ there is a neighborhood $V_j$ of $x_j$ in $\alpha\times[-1,0]$ such that $h(V_j)$ is entirely contained in $\mb{C}_{\ep_0}(\gamma_j)$. 

For each $x_j$, let $\eta_j\subset A$ be a properly embedded arc contained in $V_j$ surrounding $x_j$ going from $\alpha_j$ to $\alpha_{j+1}$ (see Figure \ref{fig:homotopy}). Denote by $a_j\in\alpha_j$ and $b_j\in\alpha_{j+1}$ the endpoints of $\eta_j$. Let $\beta_j$ be the subarc of $\alpha_j$ with endpoints given by the terminal endpoint $b_{j-1}$ of $\eta_{j-1}$ and the initial endpoint $a_j$ of $\eta_j$. As the concatenation $\beta_1\star\eta_1\star\ldots\star\beta_n\star\eta_n$ is a simple closed curve of $A$ parallel to $\alpha\times\{0\}$ (see Figure \ref{fig:homotopy}), its image
\[
h(\beta_1)\star h(\eta_1)\star\ldots\star h(\beta_n)\star h(\eta_n)
\]
is homotopic to $\alpha^*=h(\alpha\times\{0\})$. Note also that, thanks to the presence of the disk $D$, the arc $\xi_A$ is homotopic in $A$ relative to the endpoints to the concatenation
\[
\xi_A\simeq\delta_{n}\star\eta_n\star\ldots\star\eta_{s-1}\star\delta_s
\]
where $\delta_{n},\delta_s$ are terminal and initial subsegments of $\beta_n,\beta_s$ (recall that the vertices of $A$ inside $D$ are exactly $x_1,\cdots,x_s$).

We obtained a representation of $\alpha$ as a concatenation $h(\beta_1\star\eta_1\star\ldots\star\beta_n\star\eta_n)$. Let us move to the next step of the proof. Following the homotopy between $f$ and $\iota$ and using the fact that $h(\eta_j)$ is contained in the cuspidal part, we now push this concatenation to a concatenation $\beta'_1\star\eta'_1\star\ldots\star\beta'_n\star\eta'_n$ representing $\alpha$ on $S\times\{-1\}$ where each $\beta_j'$ is a subsegment of the leaf $\ell_j$ of $\lambda$ and each $\eta_j'$ is contained in a component of a tubular neighborhood of $\gamma_Y$. 

We proceed as follows. Recall that the restriction of $h$ to $\alpha_j$ coincides with $f\circ\phi_j$ where $\phi_j:\alpha_j\to\ell_j$ is an orientation preserving homeomorphism. Also recall that $f(\ell_j)$ is properly homotopic $\iota(\ell_j)$ via the restriction to $\ell_j$ of the proper homotopy between $f:S-\gamma_Y\to Q$ and $\iota:S-\gamma_Y\to\bar{Q}$. In particular, by properness of this homotopy and adjusting the segments $\eta_j$ (putting their endpoints $a_j,b_j$ closer to $x_j$ if needed), we make sure that the tracks of the points in $\phi_j(b_{j-1}),\phi_j(a_j)\in\ell_j$ under the proper homotopy between $f(\ell_j)$ and $\iota(\ell_j)$ stay inside $\mb{C}_{\ep_0}(\gamma_j)$. We denote by 
\[
\theta_{j-1}^+,\theta_j^-\subset\mb{C}_{\ep_0}(\gamma_j)
\]
the tracks $\phi_j(b_{j-1}),\phi_j(a_j)$ under the above homotopy. The composition 
\[
\bar{\theta}_{j-1}^+\star f\phi_j(\beta_j)\star\theta_j^-=\bar{\theta}_{j-1}^+\star h(\beta_j)\star\theta_j^-
\]
(where $\bar{\theta}$ denotes the inverse of the path $\theta$) is homotopic relative to the endpoints to the inclusion $\iota(\phi_j(\beta_j))\to\bar{Q}$. Now note that, by our choices, the composition 
\[
\bar{\theta}_j^-\star h(\eta_j)\star\theta_j^+,
\]
is entirely contained in $\mb{C}_{\ep_0}(\gamma_j)$ which is the intersection of a collar neighborhood of $A_{\gamma_j}\times\{-1\}$ with $S\times(-1,1)$ (see \ref{subsec:identifications}). In particular, it is homotopic relative to the endpoints to an arc $\eta'_j\subset A_{\gamma_j}\times\{-1\}$. Thus, we have that
\begin{align*}
&h(\beta_1)\star h(\eta_1)\star\ldots\star h(\beta_n)\star h(\eta_n)\\
&=f\phi_1(\beta_1)\star h(\eta_1)\star\ldots\star f\phi_n(\beta_n)\star h(\eta_n)\\
&=(\bar{\theta}_n^+\star f\phi_1(\beta_1)\star\theta_1^-)\star(\bar{\theta}_1^- h(\eta_1)\star\theta_1^+)\star\ldots\star(\bar{\theta}_{n-1}^+\star f\phi_n(\beta_n)\star\theta_n^-)\star(\bar{\theta}_n^-\star h(\eta_n)\star\theta_n^+)
\end{align*}
is homotopic to
\[
\iota(\phi_1(\beta_1))\star\iota(\eta'_1)\star\ldots\star\iota(\phi_n(\beta_n))\star\iota(\eta'_n)=\iota(\phi_1(\beta_1)\star\eta'_1\star\ldots\star\phi_n(\beta_n)\star\eta'_n).
\]
As $h(\beta_1)\star h(\eta_1)\star\ldots\star h(\beta_n)\star h(\eta_n)$ is homotopic to $\alpha^*$ and the inclusion of $S$ in $S\times[-1,1]$ is $\pi_1$-injective, we conclude that
\[
\phi_1(\beta_1)\star\eta'_1\star\ldots\star\phi_n(\beta_n)\star\eta'_n
\]
is homotopic to $\alpha$. Observe that each $\phi_j(\beta_j)$ is disjoint from $\gamma_Y$ (it is a subsegment of $\ell_j$) and each $\eta'_j$ is contained in an annular neighborhood $A_{\gamma_j}$ of a component $\gamma_j\subset\gamma_Y$ (with endpoints on opposite sides of $\gamma_j$). In particular, up to homotopies relative to the endpoints, each $\eta'_j$ intersects $\gamma_j$ in exactly one point. Thus, the intersections of the concatenation with $\gamma_Y$ occur exactly at $\eta'_j\cap\gamma_j$.  

We reached the third and last step of the proof. Using the fact that $h(\xi_A)$ is homotopic to $f(\xi_Y)$ we will show that in the concatenation $\phi_1(\beta_1)\star\eta'_1\star\ldots\star\phi_n(\beta_n)\star\eta'_n$ representing $\alpha$ on $S\times\{-1\}$ we can replace the concatenation of certain segments with $\xi_Y$ via a homotopy getting rid of $n-s$ intersections with $\gamma_Y$. This will provide the desired contradiction as we found a representative of $\alpha$ on $S$ that intersects $\gamma_Y$ less than $n=i(\alpha,\gamma_Y)$ times.

We proceed as follows. Denote by $\theta_{\rm in},\theta_{\rm out}$ the tracks of the initial and terminal endpoints of $\phi_n(\delta_n)$ and $\phi_{s+1}(\delta_{s+1})$ under the proper homotopy between $f:S-\gamma_Y\to Q$ and $\iota:S-\gamma_Y\to\bar{Q}$. We have that
\[
\bar{\theta}_{\rm in}\star f\phi_n(\delta_n)\star\theta_n^-\quad\text{\rm and }\quad\bar{\theta}_s^+\star f\phi_{s+1}(\delta_{s+1})\star\theta_{\rm out}
\]
are homotopic relative to the endpoints to $\iota(\phi_n(\delta_n))$ and $\iota(\phi_{s+1}(\delta_{s+1}))$. Recall that, by our initial assumption, it is possible to homotope $h(\xi_A)$ to $f(\xi_Y)$ relative to the endpoints. Thus
\[
h(\xi_A)\star f(\bar{\xi}_Y)=(\bar{\theta}_{\rm in}\star h(\xi_A)\star{\theta}_{\rm out})\star(\theta_{\rm out}\star f(\bar{\xi}_Y)\star\theta_{\rm in})
\]
is null-homotopic (again $\bar{\theta},\bar{\xi}$ denote the inverses of $\theta,\xi$). The second path $\bar{\theta}_{\rm out}\star f(\bar{\xi}_Y)\star\theta_{\rm in}$ is homotopic relative to the endpoints to $\iota(\xi_Y)$ (as before, the homotopy is provided by the restriction of the one between $f$ and $\iota$). The first path $\bar{\theta}_{\rm in}\star h(\xi_A)\star\theta_{\rm out}$ is homotopic relative to the endpoints to
\begin{align*}
&\bar{\theta}_{\rm in}\star h(\delta_n)\star\theta_n^-\\
 &\star(\bar{\theta}_n^-\star h(\eta_n)\star\theta_n^+)\star(\bar{\theta}_n^+\star h(\beta_1)\star\theta_1^-)\\
 &\star\ldots\\
 &\star(\bar{\theta}_{s-1}^-\star h(\eta_{s-1})\star\theta_{s-1}^+)\star(\bar{\theta}_{s-1}^-\star h(\delta_s)\star\theta_{\rm out}).
\end{align*}
In turn, as $h=f\circ\phi_j$ on $\beta_j\subset \alpha_j$, the latter coincides with
\begin{align*}
&\bar{\theta}_{\rm in}\star f\phi_n(\delta_n)\star\theta_n^-\\
 &\star(\bar{\theta}_n^-\star h(\eta_n)\star\theta_n^+)\star(\bar{\theta}_n^+\star f\phi_1(\beta_1)\star\theta_1^-)\\
 &\star\ldots\\
 &\star(\bar{\theta}_{s-1}^-\star h(\eta_{s-1})\star\theta_{s-1}^+)\star(\bar{\theta}_{s-1}^-\star f\phi_s(\delta_s)\star\theta_{\rm out})
\end{align*}
which is homotopic relative to the endpoints to
\begin{align*}
&\simeq\iota(\phi_n(\delta_n))\star\iota(\eta'_n)\star\iota(\phi_1(\beta_1))\star\ldots\star\iota(\eta'_s)\star\iota(\phi_s(\delta_s))\\
&=\iota(\phi_n(\delta_n)\star\eta'_n\star\phi_1(\beta_1)\star\ldots\star\eta'_{s-1}\star\phi_s(\delta_s)).
\end{align*}

As $\bar{\theta}_{\rm out}\star f(\bar{\xi}_Y)\star\theta_{\rm in}$ is homotopic relative to the endpoints to $\iota(\bar{\xi}_Y)$, we get that 
\[
\iota(\phi_n(\delta_n)\star\eta'_n\star\phi_1(\beta_1)\star\ldots\star\eta'_{s-1}\star\phi_s(\delta_s)\star\bar{\xi}_Y)
\]
is null-homotopic in $\bar{Q}$. Since the inclusion of $S$ in $S\times[-1,1]$ is $\pi_1$-injective, we conclude that the concatenation is also null-homotopic in $S$ and hence the two paths 
\[
\phi_n(\delta_n)\star\eta'_n\star\phi_1(\beta_1)\star\ldots\star\eta'_{s-1}\star\phi_s(\delta_s)\simeq\xi_Y
\]
are homotopic relative to the endpoints.

We are ready to finish the proof. By the above discussion, we can homotope
\[
(\phi_n(\delta_n)\star\eta_n'\star\phi_1(\beta_1)\star\ldots\star\eta'_{s-1}\star\phi_s(\delta_s))\star\phi_n(\bar{\delta}_s)\star\phi_s(\beta_s)\star\ldots\star\phi_n(\beta_n)\star\phi_n(\bar{\delta}_n)
\]
(representing $\alpha$ on $S$) to $(\xi_Y)\star\phi_s(\bar{\delta}_s)\star\phi_s(\beta_s)\star\ldots\star\phi_n(\beta_n)\star\phi_n(\bar{\delta}_n)$. However, as $\xi_Y\subset S-\gamma_Y$ the concatenation $\xi_Y\star\phi_{s+1}(\beta_{s+1})\star\eta'_{s+1}\star\ldots\star\phi_n(\beta_n)\star\eta_n'$ can intersect $\gamma_Y$ only at $\eta'_j\cap\gamma_j$, so in at most $n-s$ points. Since the concatenation is homotopic to $\alpha$, this contradicts the fact that $i(\alpha,\gamma)=n$.
\end{proof}

\subsection{Detours in the thin part}
We now analyze the behavior with respect to the thick-thin decompositions of $(S-\gamma_Y,\sigma)$ and $Q$ of the leaves of $\partial A$ that are fellow-traveling in $A$. 

More precisely, consider a leaf $\ell_Y\subset\partial A$ contained in $Y$ and a leaf $\ell\subset\partial A$ distinct from $\ell_Y$. Denote by $\tau_Y,\tau$ their non-cuspidal parts. We show that in the hyperbolic metric $(A,\sigma_A)$ these leaves cannot stay $\ep_2$-close for a time longer than $T$ for suitable choices of $\ep_2$ and $T$ (to be determined in the argument below). 

Suppose that $\tau_Y,\tau$ stay $\ep_2$-close in $A$ along subsegments $\overline{\tau}_Y,\overline{\tau}$ of length at least $T$ centered around $y_Y,y$. We first show that if one of the two segments intersects a Margulis tube in $Q$, then both segments intersect the same standard collar in $Y$.

\begin{lem}
\label{lem:ep3thin}
Let $\ep_1=\ep_0^{10}/2^{37}\pi^{8}|\chi(S)|^{16}$ be as in Lemma \ref{lem:ep3thick}. Choose $\ep_3=\ep_1/4$. If $\ep_2<\ep_1/4$ then we have the following. Suppose that there is a segment $\xi_A\subset A$ joining $x_Y\in\overline{\tau}_Y$ to $x\in\overline{\tau}$ of length $\le\ep_2$. If either $f(x_Y)$ or $f(x)$ are in the $\ep_3$-thin part $Q_{\ep_3}$ then $\ell_Y,\ell$ are both contained in $Y$ and $x,x_Y\in{\rm\bf collar}(\beta,\ep_0)$ for a closed geodesic $\beta\subset Y$.
\end{lem}

\begin{proof}
As the argument is symmetric we will only discuss the case of $\overline{\tau}_Y$. Note that a component of the $\ep_3$-thin part $Q_{\ep_3}$ is either a Margulis tube $\mb{T}_{\ep_3}(\beta)$ around a closed geodesic $\beta$ or a rank one cuspidal neighborhood $\mb{C}_{\ep_3}(\beta)$ for some accidental parabolic $\beta\subset S$ .

Suppose that $f(x_Y)\in Q_{\ep_3}$ (and lies in a component $\mb{T}_{\ep_3}(\beta)$ or $\mb{C}_{\ep_3}(\beta)$). By Lemma \ref{lem:thick to thick}, necessarily we have $x_Y\in(Y,\sigma)_{\ep_0}$ in a component of the form ${\rm\bf collar}(\beta',\ep_0)$ where $\beta'$ is homotopic to $\beta$ (recall that $\tau_Y$ does not intersect the cusps).

By assumption, $x\in\overline{\tau}$ is a point at distance at most $\ep_2$ from $x_Y$ in $A$. Thus $f(x),f(x_Y)$ are at distance at most $\ep_2$ (as $h:(A,\sigma_A)\to Q$ is 1-Lipschitz). By our choices of Margulis constants, as $2(\ep_2+\ep_3)<\ep_1$, the $\ep_2$-neighborhood of a point in $Q_{\ep_3}$ is still contained in $Q_{\ep_1}$ in same component as $f(x_Y)$. Again, by Lemma \ref{lem:thick to thick}, only the $\ep_0$-thin part of $(S-\gamma_Y,\sigma)$ can enter the $\ep_1$-thin part of $Q$. Thus $x\in(S-\gamma_Y,\sigma)_{\ep_0}$ as well in a component of the form ${\rm\bf collar}(\beta'',\ep_0)$ where $\beta''$ is homotopic to $\beta$. In particular, $\beta',\beta''$ are both homotopic to $\beta$ up to passing to the inverse. As $Q$ is homotopy equivalent to $S$, this can only happen if $\beta',\beta''$ are homotopic in $S$ and hence $\ell_Y,\ell$ are both contained in the component $Y$. By the above discussion, we have $x_Y\in{\rm\bf collar}(\beta,\ep_3)$ and $x\in{\rm\bf collar}(\beta,\ep_0)$.    
\end{proof}

We now argue that, if $\ep_2$ is smaller than an explicit threshold, the above lemma allows us to produce a shortcut in the ideal concatenation $\ell_1\star\ldots\star\ell_n$ (violating thus Lemma \ref{lem:no homotopy rel endpoints}). We deduce that $\overline{\tau}_Y,\overline{\tau}$ cannot meet the $\ep_3$-thin part $Q_{\ep_3}$ (and hence they also avoid $\ep_3$-standard collars in $Y$).

\begin{lem}
\label{lem:short bridge1} 
Let $\ep_1=\ep_0^{10}/2^{37}\pi^{8}|\chi(S)|^{16}$, $\ep_3=\ep_1/4$, and $\ep_2<\ep_3$ be as above. Suppose that there is a segment $\xi_A\subset A$ joining $x_Y\in\overline{\tau}_Y$ to $x\in\overline{\tau}$ of length $\le\ep_2$. Then both $f(x_Y),f(x)$ are contained in the complement of the $\ep_3$-thin part $Q_{\ep_3}$. 
\end{lem}

Note before we start the proof that, since $f:(S-\gamma_Y,\sigma)\to Q$ is 1-Lipschitz and $\pi_1$-injective, the above implies that $x_Y,x$ are not contained in the $\ep_3$-thin part of $(S-\gamma_Y,\sigma)$.

\begin{proof}
We proceed by contradiction and assume that one of the points $f(x_Y),f(x)$ is contained in $Q_{\ep_3}$ in a component of the form $\mb{T}_{\ep_3}(\beta)$ (or $\mb{C}_{\ep_3}(\beta)$). Again, the argument is symmetric, so we will discuss only the case of $x_Y\in\overline{\tau}_Y$.

By Lemma \ref{lem:ep3thin}, if $f(x_Y)\in Q_{\ep_3}$, then $\overline{\tau}$ is contained in $Y$ and both $x,x_Y$ are contained in ${\rm\bf collar}(\beta,\ep_0)$ for some closed geodesic $\beta\subset Y$.

Let $\xi_Y\subset{\rm\bf collar}(\beta,\ep_0)$ be an arc joining $x_Y$ to $x$ on ${\rm\bf collar}(\beta,\ep_0)$. The image of the concatenation $\xi_A\star\xi_Y$ is contained in $\mb{T}_{\ep_0}(\beta)$ (or $\mb{C}_{\ep_0}(\beta)$) because $\xi_Y$ is in ${\rm\bf collar}(\beta,\ep_0)$ and $\xi_A$ has length $\le\ep_2$ and has an endpoint in the $\ep_3$-thin part. Hence $\xi_A\star\xi_Y$ is homotopic to $\beta^q$ for some $q\in\mb{Z}$. In particular, if $\overline{\beta}$ is a loop based at $x_Y$ homotopic to $\beta$, we have that $\overline{\beta}^{-q}\star\xi_Y$ (which is an arc in $Y$) is homotopic relative to the endpoints to $\xi_A$. This violates Lemma \ref{lem:no homotopy rel endpoints}.    
\end{proof}

We conclude the following.

\begin{cor}
\label{cor:shortbridge2}
$f(\overline{\tau}_Y),f(\overline{\tau})$ (resp. $\overline{\tau}_Y,\overline{\tau}$) are both contained in the $\ep_3$-thick part of $Q$ (resp. $(S-\gamma_Y,\sigma)$).    
\end{cor}

\begin{proof}
Recall that by our choices $\overline{\tau}_Y$ stays $\ep_2$-close to $\overline{\tau}$ in $A$. Thus for every $x_Y\in\overline{\tau}_Y$ there is an arc $\xi_A\subset A$ of length $\ep_2$ joining $x_y$ to $x\in\overline{\tau}$. By Lemma \ref{lem:short bridge1}, we get that $f(x_Y)$ (resp. $x_Y$) is not contained in $Q_{\ep_3}$ (resp. $(Y,\sigma)_{\ep_3}$).    
\end{proof}

\subsection{Recurrence in the thick part}
Once we established that the non-cuspidal segments $\tau_j\subset\ell_j$ can only $\ep_2$-fellow-travel (on $A$) in the $\ep_3$-thick part (of $Q$ and $(S-\gamma_Y,\sigma)$) we want to control the amount of time they could stay close together in $A$. We do so by exploiting the recurrence and Anosov closing Lemmas \ref{lem:recurrence} and \ref{lem:quasi-geodesic} in combination with the no-shortcuts Lemma \ref{lem:no homotopy rel endpoints}. 

\begin{lem}
\label{lem:ep3thick}  
Let $\ep_1=\ep_0^{10}/2^{37}\pi^{8}|\chi(S)|^{16}$, $\ep_3=\ep_1/4$, and $\ep_2=\ep_3/8$. Suppose that $\overline{\tau}_Y,\overline{\tau}$ stay $\ep_2$-close on $A$ and have length $T$. We have 
\[
T\le T(\pi/6)\frac{16\pi^2|\chi(Y)||\chi(S-Y)|}{{\rm vol}(\ep_3/16)^2}
\]
where $T(\pi/6)$ is the constant of Lemma \ref{lem:quasi-geodesic} and ${\rm vol}(\bullet)$ is the volume of a ball of given radius in $T^1\mb{H}^2$.
\end{lem}

\begin{proof}
Again, we argue by contradiction. Suppose that $T$ is larger than the threshold above. Consider $V_Y:=\{v_Y^1,\cdots,v_Y^m\}$ the velocities of $\overline{\tau}_Y$ sampled at times of the form $2kT(\pi/6)$ for $k\le m$ and $m=\lfloor T/2T(\pi/6)\rfloor$ where $T(\pi/6)$ is the constant of Lemma \ref{lem:quasi-geodesic}. By Lemma \ref{lem:recurrence}, we find $w_Y^1,\cdots,w_Y^n\subset V_Y$ with 
\[
n=\left\lfloor\frac{{\rm vol}(\ep_3/16)}{2\pi|\chi(Y)|}\lfloor T/2T(\pi/6)\rfloor\right\rfloor
\]
such that $d_{T^1Y}(w_Y^p,w_Y^q)\le\ep_3/8$ for every $p,q\le n$. 

Consider the corresponding set of velocities $v^j$ of $\overline{\tau}$ such that $v^j$ and $w_Y^j$ are $\ep_2$-close. Again, consider $V:=\{v^1,\cdots,v^n\}$. Again, by Lemma \ref{lem:recurrence}, since by our initial assumption we have
\[
n'=\left\lfloor\frac{{\rm vol}(\ep_3/16)}{2\pi|\chi(S-Y)|}\left\lfloor\frac{{\rm vol}(\ep_3/16)}{2\pi|\chi(Y)|}\lfloor T/2T(\pi/6)\rfloor\right\rfloor\right\rfloor\ge 2,
\]
there are at least two velocities $v^p,v^q$ such that $d_{T^1Y}(v^p,v^q)<\ep_3/8$.

Let $\eta_Y\subset\overline{\tau}_Y,\eta\subset\overline{\tau}$ be the sub-intervals between $w_Y^p,w_Y^q$ and between $v^p,v^q$ (along $\overline{\tau}_Y,\overline{\tau}$) respectively. By the above discussion, the endpoints of $\eta_Y$ and of $\eta$ are very close in $Y$ and even in the unit tangent bundle $T^1Y$ (if we consider the initial and terminal velocities). Let $\mu_Y,\mu$ be short geodesics of length $\le\ep_3/8$ joining the endpoints of $\eta_Y,\eta$ respectively and such that the initial and terminal velocities of $\mu_Y,\mu$ have angular distance at most $\ep_3$ from the corresponding velocities of $\eta_Y,\eta$.

By Lemma \ref{lem:quasi-geodesic}, the concatenations $\eta_Y\star\mu_Y$ and $\eta\star\mu$ are not null-homotopic.  

The segments $\eta_Y,\eta$ cobound a rectangle $R$ in $A$ with arcs $\xi_A,\xi_A'\subset A$ of length $\le\ep_2$ joining their initial and terminal points. Consider the concatenation $\xi_A\star\mu\star\xi_A'\star\mu_Y$. It is a loop of length $\le 2(\ep_3/8+\ep_2)=\ep_3/2$. In particular, by Corollary \ref{cor:shortbridge2}, it is null-homotopic in $Q$ (both $\overline{\tau}_Y,\overline{\tau}$ are contained in the $\ep_3$-thick part of $Q$) and we can find an immersed rectangle $R'$ bounded by $\xi_A,\mu,\xi_A',\mu_Y$. Gluing $R$ to $R'$ we find an immersed annulus $h':\mb{S}^1\times[0,1]\to R\cup R'\subset Q$ between $\eta_Y\star\mu_Y$ and $\eta\star\mu$. With a little abuse of notation we denote by $\xi_A\subset\mb{S}^1\times[0,1]$ the arc corresponding to $\xi_A$ under the chosen parameterization $h'$. Thus the two concatenations $\eta_Y\star\mu_Y$ and $\eta\star\mu$ are homotopic in $Q$. As $Q$ is homotopy equivalent to $S$, we conclude that $\eta_Y\star\mu_Y$ and $\eta\star\mu$ are homotopic also in $Y$. Let $h_Y:\mb{S}^1\times[-1,0]\to Y$ be an immersed annulus between the two.

We now produce a shortcut using the two annuli.

We glue the two annuli in a 2-torus $U:=\mb{S}^1\times[-1,1]/(-1\sim 1)$ and we glue $fh_Y$ and $h'$ to get an immersed 2-torus $fh_Y\cup h':U\to Q$. The image ${\rm Im}(\pi_1(U)\to\pi_1(Q))$ is abelian and contains the infinite cyclic subgroup generated by the (homotopic) loops $\eta_Y\star\mu_Y$ and $\eta\star\mu$, so, it is an infinite cyclic subgroup $\langle\beta\rangle$ containing $\eta_Y\star\mu_Y$. Let $\xi_Y\subset \mb{S}^1\times[-1,0]$ be an arc starting at $x_Y\in\eta_Y$ and ending at $x\in\eta$. Consider now the loop $\delta=\xi_A\star\xi_Y$. As the image of $\pi_1(U)$ coincides with an infinite cyclic subgroup $\langle\beta\rangle$ containing $\eta_Y\star\mu_Y$, we have that $\delta=\beta^q$ for some $q$. Note that, as we also have $\beta^k=\eta_Y\star\mu_Y$ for some $k$, necessarily $\beta\in\pi_1(Y)$ and hence we can find a loop $\beta_{x_Y}$ in $Y$ based at $x_Y$ representing it. As the loop $\delta\star\beta_{x_Y}^{-q}\in\pi_1(Y)$ is null-homotopic we have that $\xi_A$ is homotopic relative to the endpoints to $\xi_Y\star\beta_{x_Y}^{-q}$. This contradicts again Lemma \ref{lem:no homotopy rel endpoints}.    
\end{proof}

\subsection{Length-area argument}
We now have all the ingredients to find one non-cuspidal segment $\tau_j\subset\ell_j$ contained in $Y$ and of moderate length.

Define a {\em $\ep_2$-rectangle} between $\tau_i\subset\ell_i,\tau_j\subset\ell_j$ to be the region of $(A,\sigma_A)$ between the two segments where the distance in $A$ between the two is smaller than $\ep_2$. The sides of a rectangle contained in $\tau_i,\tau_j$ are the {\em horizontal sides} and the other two are the {\em vertical sides}.

Let $\mc{R}$ be the collection of $\ep_2$-rectangles in $A$ whose boundary consists of two horizontal segments contained one in the non-cuspidal part $\tau_Y$ of a leaf $\ell_Y$ of $\partial A\cap Y$ and the other in a different leaf of $\partial A$. Note that each $R\in\mc{R}$ separates $A$ into two components, one of which is a disk containing at least one of the ideal vertices of $A$ and the other is an annulus. So the first $\ep_2$-rectangle splits non-trivially the set of ideal vertices of $A$ into two non-empty sets. The second $\ep_2$-rectangle further splits non trivially one of these two sets into two smaller sets, and so on and so forth. It follows that the number of $\ep_2$-rectangles, denoted by $|\mc{R}|$, is at most the number of ideal vertices $n$. 

For every leaf $\ell_Y\subset\partial A\cap Y$ denote by $n(\ell_Y)$ the number of $\ep_2$-rectangles $R\in\mc{R}$ with one horizontal side on $\ell_Y$ and the other horizontal side on a leaf of $\partial A-Y$ and by $m(\ell_Y)$ the number of $\ep_2$-rectangles with a horizontal side on $\ell_Y$ the other horizontal side on a leaf of $\partial A\cap Y$. We have
\[
n\ge|\mc{R}|=\sum_{\ell_Y\subset\lambda\cap Y}{\left(n(\ell_Y)+\frac{1}{2}m(\ell_Y)\right)}.
\]

Since the number of leaves in $\partial A$ which lie in $Y$ is at least $n/2$ we have that at least half of them satisfy $n(\ell_Y)+\frac{1}{2}m(\ell_Y)\le 4$. Denote by $\mc{L}$ such subset. It has $|\mc{L}|\ge n/4$. For every leaf $\ell_Y\in\mc{L}$ denote by $\kappa(\ell_Y)$ the union of the subsegments of the $\ep_2$-rectangles in $\mc{R}$ with a horizontal side in $\ell_Y$. By Lemma \ref{lem:ep3thick}, each of them has length at most $T=(16\pi^2T(\pi/6)|\chi(S)|^2)/({\rm vol}(\ep_3/16)^2)$, so the total length is at most
\[
\ell(\kappa(\ell_Y))\le (n(\ell_Y)+m(\ell_Y))T\le \frac{128T(\pi/6)\pi^2|\chi(S)|^2}{{\rm vol}(\ep_3/16)^2}.
\]

When $\alpha$ is not parabolic we select an additional collection of subsegments. For each leaf $\ell_Y\in\mc{L}$ we consider also the connected subsegment $\kappa_\alpha(\ell_Y)$ consisting of those points at distance at most $\ep_2/2$ from $\alpha^*\subset\partial A$. Each $\kappa_\alpha(\ell_Y)$ cobounds a $\ep_2/2$-rectangle with a subsegment $\kappa_\alpha'(\ell_Y)\subset\alpha^*$ whose vertical sides have length $\le\ep_2/2$. Note that if $\ell_Y,\ell_Y'$ are distinct leaves then $\kappa_\alpha'(\ell_Y),\kappa_\alpha'(\ell_Y')$ are disjoint. Thus, we have
\[
\sum_{\ell_Y\in\mc{L}}{{\rm Length}(\kappa_\alpha(\ell_Y))}=\sum_{\ell_Y\in\mc{L}}{{\rm Length}(\kappa_\alpha'(\ell_Y))}\le\ell(\alpha^*)=L.
\]

The complement $\tau_Y-(\kappa(\ell_Y)\cup\kappa_\alpha(\ell_Y))$ has distance at least $\ep_2$ from $\partial A-\ell_Y$. Hence, it has an embedded normal neighborhood $N(\ell_Y)$ of radius $\ep_2/2$. Note that if $\ell_Y,\ell_Y'$ are distinct leaves in $\partial A$ then their corresponding subsets $N(\ell_Y),N(\ell_Y')$ are disjoint otherwise we would have a point in $\tau_Y-(\kappa(\ell_Y)\cup\kappa_\alpha(\ell_Y))$ and a point in $\tau_Y'-(\kappa(\ell_Y')\cup\kappa_\alpha(\ell_Y'))$ at distance at most $\ep_2$. Therefore
\begin{align*}
n\pi={\rm Area}(A) &\ge{\rm Area}\left(\bigcup_{\ell_Y\in\mc{L}}{N(\ell_Y)}\right)\\
&=\sum_{\ell_Y\in\mc{L}}{{\rm Area}(N(\ell_Y))}\\
&=\sum_{\ell_Y\in\mc{L}}{{\rm Length}(\tau_Y-(\kappa(\ell_Y)\cup\kappa'(\ell_Y)))\sinh(\ep_2/2)}
\end{align*}
Since the number of summands is at least $n/4$ we conclude that at least half of them have length bounded by 
\[
{\rm Length}(\tau_Y-(\kappa(\ell_Y)\cup\kappa'(\ell_Y)))\le\frac{8\pi}{\sinh(\ep_2/2)}.
\]
Hence, for those segments we have
\begin{align*}
\ell(\tau_Y) &={\rm Length}(\tau_Y-(\kappa(\ell_Y)\cup\kappa'(\ell_Y)))+{\rm Length}(\kappa(\ell_Y))+{\rm Length}(\kappa'(\ell_Y))\\
 &\le\frac{8\pi}{\sinh(\ep_2/2)}+\frac{128T(\pi/6)\pi^2|\chi(S)|^2}{{\rm vol}(\ep_3/16)^2}+L. 
\end{align*}

By definition each endpoint of $\tau_Y$ lies on the boundary of the $\ep_0$-cuspidal part of $S-\gamma_Y$ and hence has a loop around it representing the core of the cusp of length $\le 2\ep_0$. By definition of subsurface projection, there exists a representative of $\pi_Y(\alpha)$ which is a concatenation of one or two copies of $\tau_Y$ with the two loops at its endpoints. In particular such concatenation has length at most
\[
4\ep_0+\frac{16\pi}{\sinh(\ep_3/16)}+\frac{256T(\pi/6)\pi^2|\chi(S)|^2}{{\rm vol}(\ep_3/16)^2}+2L
\]
as desired. This finishes the proof of Theorem \ref{thm:effeff}.

\section{Effective convex core width}
\label{sec:effcusp}

In this section we discuss the second ingredient of the proof of Theorem \ref{thm:main1}. Let $Q$ be a hyperbolic manifold diffeomorphic to ${\rm int}(Y)\times\mb{R}$ such that each component $\theta\subset\partial Y$ is parabolic. Recall that $Q$ has a convex core ${\rm CC}(Q)\subset Q$ which, in case it has finite volume, is a codimension 0 submanifold isotopic to ${\rm int}(Y)\times[-1,1]-(P^+\times\{1\}\cup P^-\times\{-1\})$ (where $P^\pm\subset{\rm int}(Y)\times\{\pm1\}$ are multicurves representing the accidental parabolics). The main result of the section is the following effective width estimate. 

\begin{thm}
\label{thm:effcusp}
Let $\ep\le\ep_0$ be a Margulis constant for dimensions 2 and 3. Let $Y$ be a compact orientable surface with $\chi(Y)<0$. Consider a hyperbolic manifold $Q$ diffeomorphic to ${\rm int}(Y)\times\mb{R}$ such that each component $\theta\subset\partial Y$ is parabolic and with a finite volume convex core ${\rm CC}(Q)$. Let $\kappa\subset{\rm CC}(Q)-\bigcup_{\theta\subset\partial Y}\partial\mb{C}_{\ep}(\theta)$ be an arc joining two boundary components of ${\rm CC}(Q)$ on opposite sides of ${\rm CC}(Q)$. Suppose that there are non-peripheral simple closed curves $\alpha,\beta\subset Y$ whose geodesic representatives in $Q$ have length at most $L\ge2\log(256\pi^2|\chi(Y)|^4/\ep^2)$. Then
\[
d_{\mc{C}(Y)}(\alpha,\beta)\le (2+2\log_2(Le^{L/2}))\frac{\sinh(2\ep+2L)-2(\ep+L)}{\sinh{2\ep}-2\ep}\ell(\kappa).
\]
\end{thm}

Theorem \ref{thm:effcusp} is a mild variation of \cite[Lemma 4.4 and Theorem 4.1]{APT} by Aougab, Patel, and Taylor, which, in turn, makes effective a theorem of Brock and Bromberg \cite[Theorem 7.16]{BB11}.

\begin{proof}[Proof of Theorem \ref{thm:effcusp}]
We proceed exactly as in \cite[Lemma 4.4]{APT}. We only need an adjustment to our setup. The first step is to use the Canary-Thurston's Filling Theorem \cite[Filling Theorem]{C96} to find a so-called {\em 1-Lipschitz sweepout} $\{f_t:({\rm int}(Y),\sigma_t)\to Q\}_{t\in[0,1]}$ connecting $\alpha^*$ to $\beta^*$ that is a proper homotopy $f:{\rm int}(Y)\times[0,1]\to Q$ together with a continuous family of complete finite area hyperbolic metrics $({\rm int}(Y),\sigma_t)$ such that $f_t:({\rm int}(Y),\sigma_t)\to Q$ is 1-Lipschitz and that the restrictions $f_0,f_1$ to $\alpha,\beta$ are length preserving parameterizations of $\alpha^*,\beta^*$ respectively. Consider the open (by continuity of $\sigma_t$) sets $I(\gamma)=\{t\in[0,1]\,|\,\ell_{\sigma_t}(\gamma_t)<L\}$ where $\gamma_t\subset Y$ is the shortest non-peripheral loop over all representatives of $\gamma$ on $Y$ (with respect to the metric $\sigma_t$) whose image under $f_t$ intersects $\kappa$. As $\kappa$ joins two boundary components of ${\rm CC}(Q)$ on opposite sides, every $f_t$ intersects $\kappa$. If $x\in Y$ is a point such that $f_t(x)\in\kappa$ then $x\not\in{\rm\bf cusp}(\theta,\ep)$ for every $\theta\subset\partial Y$ as $f_t(x)\not\in\mb{C}_\ep(\theta)$ by assumption on $\kappa$. Thus there are two possibilities. Either $x\in{\rm\bf collar}(\gamma,\ep)$ for some simple closed geodesic $\gamma$, in which case $t\in I(\gamma)$, or $x$ is contained in the $\ep$-thick part of $Y$, in which case, by Lemma \ref{lem:short gen}, there is a non-peripheral geodesic loop $\gamma$ based at $x$ of length $\le2\log(256\pi^2|\chi(Y)|^4/\ep^2)$ so $t\in I(\gamma)$. We deduce that the sets $I(\gamma)$ cover $[0,1]$. Additionally, by Lemmas \ref{lem:intersec dist} and \ref{lem:length intersection}, if $I(\gamma)\cap I(\gamma')\neq\emptyset$, then 
\[
d_{\mc{C}}(\gamma,\gamma')\le 2\log_2(i(\gamma,\gamma'))+2\le 2\log_2(Le^{L/2})+2.    
\]

Construct now a graph $G_\kappa$ as follows. The vertices are the curves $\gamma\in\mc{C}$ such that $I(\gamma)\neq\emptyset$ and we have an edge between $\gamma,\gamma'$ if $I(\gamma)\cap I(\gamma')\neq\emptyset$. Note that the collection $I(\gamma)$ is an open cover of $[0,1]$. Since the restrictions of $f_0,f_1$ to $\alpha,\beta$ are length preserving parameterizations of $\alpha^*,\beta^*$ and $\ell_Q(\alpha^*),\ell_Q(\beta^*)\le L$ by assumption, we have $I(\alpha),I(\beta)\neq\emptyset$ and hence $\alpha,\beta\in G_\kappa$. Let $\delta>0$ be a Lebesgue number of the covering. Let us show that the graph $G_\kappa$ is connected. Consider $\gamma_1,\gamma_2$ such that $I(\gamma_1),I(\gamma_2)\neq\emptyset$. Pick $t_j\in I(\gamma_j)$. We can assume without loss of generality that $t_1<t_2$. We now partition $[t_1,t_2]$ as $t_1=s_0<s_1<\cdots<s_{n-1}<s_n=t_2$ so that $s_j-s_{j-1}<\delta$. As $\delta$ is a Lebesgue number of the covering each interval $[s_{j-1},s_j]$ is contained in some $I(\overline{\gamma}_j)$ and $I(\overline{\gamma}_j)\cap I(\overline{\gamma}_{j+1})\neq\emptyset$ as they both contain $s_j$ so $\overline{\gamma}_j,\overline{\gamma}_{j+1}$ are connected by an edge in $G_\kappa$. This shows that the graph is connected. 

Split $\kappa$ into $n-1$ length 1 segments $\kappa_1,\cdots,\kappa_{n-1}$ and one last segment $\kappa_n$ of length at most 2. Let $\mc{S}_j$ be the set of curves of $Y$ which have a representative in $Q$ which is a loop of length at most $L$ based at a point in $\kappa_j$. By \cite[Lemma 3.6]{APT} we have 
\[
|\mc{S}_j|\le\frac{{\rm vol}(B_{\mb{H}^3}(L+\ep))}{{\rm vol}(B_{\mb{H}^3}(\ep))}=\frac{\sinh(2\ep+2L)-2(\ep+L)}{\sinh{2\ep}-2\ep}.    
\]

We can now bound the number of vertices of $G_\kappa$. By definition, each vertex $\gamma$ satisfies $I(\gamma)\neq\emptyset$ and, hence, $\gamma\in\mc{S}_j$ for some $\mc{S}_j$. Thus, the number of vertices is at most
\[
\sum_{j\le n}{|\mc{S}_j|}\le\ell(\kappa)\frac{\sinh(2\ep+2L)-2(\ep+L)}{\sinh{2\ep}-2\ep}.    
\]

As $G_\kappa$ is connected, there is a path in it joining $\alpha,\beta$ of length at most the number of vertices of $G_\kappa$. Putting together the previous estimates we get
\[
d_{\mc{C}(Y)}(\alpha,\beta)\le (2\log_2(Le^{L/2})+2)\frac{\sinh(2\ep+2L)-2(\ep+L)}{\sinh{2\ep}-2\ep}\ell(\kappa).
\]
This finishes the proof of the theorem.
\end{proof}

\section{Proof of Theorem \ref{thm:main1}}
\label{sec:proof}

We have all the ingredients to prove Theorem \ref{thm:main1}.

\subsection{Strategy}
Let us describe in more detail the strategy outlined in the introduction. Let $Y\subset S$ be a proper connected essential non-annular subsurface. Consider the drilled manifold 
\[
M:=S\times[-1,1]-V
\]
where $V$ is a tubular neighborhood of $\gamma_Y\times\{0\}$ and $\gamma_Y$ is the multicurve obtained from $\partial Y$ by collapsing parallel components to single ones. Note that $S\times[-1,1]$ is obtained from $M$ as a Dehn filling, that is, attaching back every solid torus component $V_\gamma\subset V$ (the tubular neighborhood of $\gamma\times\{0\}\subset\gamma_Y\times\{0\}$). Such a gluing is completely determined by the isotopy class on $\partial V_\gamma$ of the so-called {\em meridian} of $V_\gamma$, the unique simple closed curve $\mu_\gamma\subset\partial V_\gamma$ bounding a properly embedded disk in $V_\gamma$. 

So far, we only made topological observations, we now turn to geometry. By Theorem \ref{thm:hyperbolization2}, the end invariants $\nu^-,\nu^+$ of $Q$ determine a unique hyperbolic structure on $M$ with rank two cusps at $\partial V$ realizing those end invariants. We denote it by $M(\nu^-,\nu^+)$. We will argue that $Q$ is obtained as a long Hyperbolic Dehn Filling from $M(\nu^-,\nu^+)$ (see Theorem \ref{thm:filling}). The length estimates will follow from \cite{FPS22b} provided that we can effectively compute the lengths of the flat geodesic representatives in $\partial\mb{C}_{\ep_0}(\partial V_\gamma,M(\nu^-,\nu^+))$ of the homotopy class of the meridians $\mu_\gamma$.

In order to do so we use the technology that we developed in the previous sections (Theorem \ref{thm:effeff} and Theorem \ref{thm:effcusp}) combined with some covering arguments (as in \cite{FSVa,FSVb}). More precisely, we apply Theorem \ref{thm:effeff} (to the $\pi_1(S\times\{\pm1\})$-coverings of $M(\nu^-,\nu^+)$) to produce two moderate-length curves $\delta^-,\delta^+\subset Y$ which are close to $\pi_Y(\nu^-),\pi_Y(\nu^+)$. Then we feed these curves to Theorem \ref{thm:effcusp} applied to the $\pi_1(Y\times\{0\})$-cover of $M(\nu^-,\nu^+)$ and deduce that in such manifold the width of the non-cuspidal part of the convex core is at least $d_{\mc{C}(Y)}(\delta^-,\delta^+)$. A covering argument (Proposition \ref{pro:covering2}) allows us to deduce that the meridians $\mu_\gamma$ have all flat geodesic representatives of length at least $d_{\mc{C}(Y)}(\delta^-,\delta^+)$. As stated before, this bound together with Theorem \ref{thm:filling} and some standard techniques gives us the desired result.

\subsection{From ends to moderate-length curves}
\label{subsec:from ends to laminations}
As a first step we produce from the end invariants $\nu^-,\nu^+$ moderate-length curves in $M(\nu^-,\nu^+)$.

Recall that each $\nu^\pm$ is a combination of minimal filling laminations $\lambda\subset\mc{EL}(W)$ or finite area hyperbolic structures $(W,\sigma)$ on connected essential subsurfaces $W\subset S$ and a finite set $P\subset S$ of pairwise disjoint annuli (representing parabolic elements in $M$) so that the union of all such subsurfaces is the whole $S$. 

We can turn $\nu^\pm$ into a lamination on $S$ by replacing each hyperbolic structure $(W,\sigma)\subset\nu^\pm$ with a Bers pants decomposition ${\rm\bf short}(W,\sigma)$ (as given by Lemma \ref{lem:short gen}) and each annulus in $P$ with its core curve. With a little abuse of notation, we will still denote the resulting laminations by $\nu^\pm$. With this notation, we have the following.

\begin{dfn}[End Invariants and Subsurfaces]
\label{dfn:intersection}
We say that the end invariants $\nu^\pm$ intersect essentially a subsurface $Y\subset S$ if the corresponding laminations, obtained by replacing annuli with core curves and finite area hyperbolic structures with (some) Bers' pants decompositions, intersect $Y$.  
\end{dfn} 

\begin{remark}
\label{rmk:additive error}
We note that Bers' pants decompositions are not always unique. However, two different Bers' decompositions of a finite area hyperbolic surface $(W,\sigma_W)$, having length at most $2\pi|\chi(W)|$, can intersect at most $2\pi|\chi(W)|e^{\pi|\chi(W)|}$ (by Lemma \ref{lem:length intersection}). Thus, the subsurface projection of a Bers' pants decomposition of $(W,\sigma_W)$ is only well defined up to an error of $2+\log_2(2\pi|\chi(W)|e^{\pi|\chi(W)|})$ (by Lemma \ref{lem:intersec dist}).    
\end{remark}

\begin{lem}
\label{lem:end short}
There are $\alpha^-,\alpha^+\in\mc{C}(S)$ such that $\alpha^-\times\{-1\},\alpha^+\times\{1\}\subset M(\nu^-,\nu^+)$ are either parabolic or homotopic to closed geodesics of length at most $4\pi|\chi(S)|$ and such that their their subsurface projections to $Y$ satisfy 
\[
d_Y(\alpha^-,\nu^-),d_Y(\alpha^+,\nu^+)\le 4.
\]
\end{lem}

\begin{proof}
By assumption, both end invariants $\nu^\pm$ intersect essentially $Y$ and none of the components of $\partial Y$ is parabolic, that is, none of the annuli in $\nu^\pm$ has core curve isotopic to $\partial Y$. Let $\lambda^\pm$ be a component of $\nu^\pm$ intersecting essentially $Y$. Let $W^\pm\subset S$ be the component of definition of $\lambda^\pm$.

If $\lambda^\pm\in\mc{EL}(W^\pm)$ then, by Theorem \ref{thm:end to length}, there exists a sequence of simple closed curves $\alpha^\pm_n\in\mc{C}(W^\pm)$ converging to $\lambda^\pm$ in the boundary of $\mc{C}(W^\pm)$ such that each $\alpha^\pm_n\times\{\pm1\}$ is homotopic to a closed geodesic of length at most $2\cdot{\rm arccosh}(|\chi(W^\pm)|+1)<4\pi|\chi(S)|$. By Lemma \ref{lem:projection of limit}, for every $n$ large enough we have $\pi_Y(\nu^\pm)\subset\pi_Y(\alpha_n)$. 

If $\lambda^\pm\subset{\rm\bf short}(W^\pm,\sigma_{W^\pm})$ then, by Theorem \ref{thm:end to length}, $\lambda^\pm\times\{\pm1\}$ is homotopic to a closed geodesic whose length is bounded by $2\ell_\sigma(\lambda^\pm)\le 4\pi|\chi(W^\pm)|$ (by Lemma \ref{lem:short gen}). In this case $\pi_Y(\lambda^\pm)\subset\pi_Y(\nu^\pm)$.

If $\lambda^\pm$ is a core curve of an annulus in $P$, then $\lambda\times\{\pm1\}$ is parabolic and, hence, it has length 0 according to the convention adopted in Theorem \ref{thm:effcusp}. In this case $\pi_Y(\lambda^\pm)\subset\pi_Y(\nu^\pm)$.

In order to conclude it is enough to recall that $\pi_Y(\bullet)$ is a subset of diameter at most 4 of $\mc{C}(W^\pm)$.
\end{proof}

\subsection{The effective efficiency step}
Next we use Theorem \ref{thm:effeff} and pass from the moderate-length curves produced by the previous section to moderate-length curves on $Y$ representing the subsurface projections of $\nu^-,\nu^+$.

\begin{lem}
\label{lem:covering1}
There are representatives $\beta^-,\beta^+\subset Y$ of the subsurface projections $\pi_Y(\alpha^-),\pi_Y(\alpha^+)$ whose geodesic representatives in $M(\nu^-,\nu^+)$ have length at most 
\[
c_1|\chi(S)|^{98}
\]
for $c_1=2^{385}/\ep_0^{60}$.
\end{lem}

\begin{proof}
Consider the coverings $Q^\pm\to M(\nu^-,\nu^+)$ corresponding to the subgroups $\pi_1(S\times\{\pm1\})$ of $\pi_1(M)$. By Bonahon's Tameness \cite{Bo86}, they are hyperbolic manifolds diffeomorphic to $S\times\mb{R}$ with rank one cusps at $\gamma_Y$. We apply Theorem \ref{thm:effeff} and Remark \ref{rmk:length coeff} with input given by the curves $\alpha^\pm$ provided by Lemma \ref{lem:end short} and obtain as output $\beta^\pm\subset\pi_Y(\alpha^\pm)$ of length
\[
\ell(\beta^\pm)\le 2\ell(\alpha^\pm)+\frac{2^{384}}{\ep_0^{60}}|\chi(S)|^{98}\le 8\pi|\chi(S)|+\frac{2^{384}}{\ep_0^{60}}|\chi(S)|^{98}<\frac{2^{385}}{\ep_0^{60}}|\chi(S)|^{98}.
\]
\end{proof}

We can refine $\beta^\pm$ to much shorter curves $\delta^\pm$ up to paying a small cost in terms of curve graph distances. 

\begin{lem}
\label{lem:shorter}
There are $\delta^-,\delta^+\in\mc{C}(Y)$ whose geodesic representatives in $M(\nu^-,\nu^+)$ have length at most $2\cdot{\rm arccosh}(|\chi(Y)|+1)$ and such that
\[
d_{\mc{C}(Y)}(\delta^\pm,\beta^\pm)\le c_2\log|\chi(S)|. 
\]
for some universal constant $c_2>0$.    
\end{lem}

\begin{remark}
\label{rmk:shorter}
Numerically, we can choose $c_2=570\log_2(c_1)$.
\end{remark}

\begin{proof}
By Lemma \ref{lem:covering1}, there are curves $\beta^-,\beta^+\subset Y$ representing the subsurface projections $\pi_Y(\alpha^-),\pi_Y(\alpha^+)$ whose geodesic representatives in $Q_Y$ have length at most $c_1|\chi(S)|^{98}$. Let $f^\pm:({\rm int}(Y),\sigma^\pm)\to Q_Y$ be a pleated surface homotopic to the inclusion of ${\rm int}(Y)$ mapping $\beta^\pm$ geodesically. Consider a non-peripheral $\delta^\pm$ of length at most $2\cdot{\rm arccosh}(|\chi(Y)|+1)$ for $({\rm int}(Y),\sigma^\pm)$ as provided by Lemma \ref{lem:short gen}. As $f^\pm$ is 1-Lipschitz, we also have $\ell_Q(\delta^\pm)\le 2\cdot{\rm arccosh}(|\chi(Y)|+1)$. The distance in $\mc{C}(Y)$ between $\delta^\pm,\beta^\pm$ can be bounded using Lemmas \ref{lem:intersec dist} and \ref{lem:length intersection}.
\begin{align*}
d_Y(\delta^\pm,\beta^\pm) &\le 2+2\log_2(i(\delta^\pm,\beta^\pm)) &\\
 &\le 2+2\log_2(\ell(\beta^\pm)e^{\ell(\delta^\pm)/2}) &\\
 &\le 2+2\log_2(c_1|\chi(S)|^{98})+\log_2(e)\cdot{\rm arccosh}(|\chi(Y)|+1) &\\
 &\le 2+2\log_2(c_1)\log_2(e)\log(|\chi(S)|)^{98})+\log_2(e)\log(4|\chi(S)|).
 \end{align*}  
The conclusion follows with a little bit of calculus.
\end{proof}

\subsection{The effective cusp size step}
Lastly, we use Theorem \ref{thm:effcusp} combined with a covering argument to bound from below the length of the flat geodesic representative of the meridian $\mu_\gamma$ on $\partial\mb{C}_{\ep_0}(\partial V_\gamma)$. 

\begin{pro}
\label{pro:covering2}
Let $\ep_Y:=\ep_0^{10}/2^{37}\pi^8|\chi(Y)|^{16}$. For every component $\gamma\subset\gamma_Y$ and every slope $\mu\subset\partial\mb{C}_{\ep_Y}(\partial V_\gamma)$ not conjugate into $\pi_1(Y)$, the flat geodesic representative $\overline{\mu}$ of $\mu$ on $\partial\mc{C}_{\ep_Y}(\partial V_\gamma)$ has length at least 
\[
c_3\frac{d_{\mc{C}(Y)}(\delta^-,\delta^+)}{|\chi(Y)|^{264}}-36\log(|\chi(Y)|)-c_4
\]
for some universal constants $c_3,c_4>0$.
\end{pro}

\begin{remark}
\label{rmk:meridian}
Numerically, we can choose $c_3=\ep_0^{150}/2^{870}$ and $c_4=40\log(64/\ep_0)$.
\end{remark}

\begin{proof}
Consider the covering $p_Y:Q_Y\to M(\nu^-,\nu^+)$ corresponding to $\pi_1({\rm int}(Y)\times\{0\})$. By Bonahon's Tameness \cite{Bo86}, it is a hyperbolic manifold diffeomorphic to ${\rm int}(Y)\times\mb{R}$ (where each component $\theta\subset\partial Y$ is parabolic). By Canary-Thurston's Covering Theorem \cite{C96}, its convex core ${\rm CC}(Q_Y)$ has finite volume and, hence, is isotopic to ${\rm int}(Y)\times[-1,1]-(P^+\times\{1\}\cup P^-\times\{-1\}$ (where $P^\pm\subset Y\times\{\pm1\}$ are the multicurves representing the accidental parabolics of $Q_Y$) and intersects the boundary of each $\ep$-cuspidal neighborhood in a compact annulus. The surface ${\rm int}(Y)\times\{0\}$ lifts diffeomorphically to $Q_Y$ and we use it as a marking of $\pi_1(Q_Y)$. By Lemma \ref{lem:shorter} the geodesic representatives (in $M(\nu^-,\nu^+)$ and hence) in $Q_Y$ of the curves $\delta^\pm$ have length at most $2\cdot{\rm arccosh}(|\chi(Y)|+1)$. We first apply Theorem \ref{thm:effcusp} to deduce that the width of the $\ep$-non-cuspidal part of the convex core ${\rm CC}(Q_Y)$ is large. Denote by $\partial_0{\rm CC}(Q_Y),\partial_1{\rm CC}(Q_Y)$ the two (possibly disconnected if there are accidental parabolics) sides of the convex core ${\rm CC}(Q_Y)$.

\begin{lem}
\label{lem:dist comp}
Let $\ep_Y\le\ep\le\ep_0$. Let $Q_0:={\rm CC}(Q_Y)-\bigcup_{\theta\subset\partial Y}{\mb{C}_\ep(\theta)}$ be the $\ep$-non-cuspidal part of the convex core ${\rm CC}(Q_Y)$. The distance in $Q_0$ between a boundary component $C_0\subset\partial_0{\rm CC}(Q_Y)\cap Q_0$ and a boundary component of $C_1\subset\partial_1{\rm CC}(Q_Y)\cap Q_0$ is at least
\[
\frac{c_5}{|\chi(Y)|^{264}}d_{\mc{C}(Y)}(\delta^-,\delta^+)
\]
for some universal constant $c_5>0$.
\end{lem}

\begin{proof}
By Theorem \ref{thm:effcusp}, the distance between $C_0,C_1$ is at least
\[
\frac{1}{2\log_2(Le^{L/2})+2}\cdot\frac{\sinh{2\ep}-2\ep}{\sinh(2\ep+2L)-2(\ep+L)}d_{\mc{C}(Y)}(\delta^-,\delta^+)
\]
with $L=4\cdot{\rm arccosh}(|\chi(Y)|+1)+2\log(4\pi|\chi(Y)|/\ep^2)$. Using the approximation $(\sinh{2\ep}-2\ep)/(\sinh(2\ep+2L)-2(\ep+L))\ge\ep^3/\sinh(2\ep+2L)\ge 2\ep^3/e^{2\ep+2L}$ (as in \cite[Remark 3.7]{APT}) and the fact that $(2\log_2(Le^{L/2})+2)e^{2\ep+2L}\le 4e^{3L}$ we get
\[
\frac{1}{2\log_2(Le^{L/2})+2}\cdot\frac{\sinh{2\ep}-2\ep}{\sinh(2\ep+2L)-2(\ep+L)}\ge\frac{\ep^3}{2e^{3L}}.
\]
Combining with $L=2\log(256\pi^2|\chi(Y)|^4/\ep^2)$ we get the lower bound
\[
\frac{\ep^3}{2e^{3L}}\ge\frac{\ep^{15}}{2^{49}\pi^{12}|\chi(Y)|^{24}}.
\]
The conclusion follows from some calculus and $\ep\ge\ep_Y=\ep_0^{10}/2^{37}\pi^8|\chi(Y)|^{16}$.
\end{proof}

\begin{remark}
\label{rmk:distance cusp}
Numerically, we can choose $c_5=\ep_0^{150}/2^{868}$.
\end{remark}

We now argue that the fact that the width of $Q_0\subset{\rm CC}(Q_Y)$ is large implies that every slope of $\partial\mb{C}_{\ep_Y}(\partial V_\theta)$ not homotopic to $\theta$ is long. The heuristic idea is that a loop on $\partial\mb{C}_{\ep_Y}(\partial V_\gamma)$ representing a slope that is not homotopic into $\partial Y$ suitably lifts to an arc in $Q_Y$ that intersects two boundary components of ${\rm CC}(Q_Y)$ on opposite sides of ${\rm CC}(Q_Y)$. So it must be long.

In order to understand the lifts of a slope on $\partial\mb{C}_{\ep_Y}(\partial V_\gamma)$ to $Q_Y$, we need a good understanding of the pre-image $p_Y^{-1}(\mb{C}_{\ep_Y}(\partial V_\gamma))$. This depends on the algebraic structure of the fundamental group, in particular on the interaction between $\pi_1(\partial V_\gamma)$ and $\pi_1(Y)$. The main properties that we need are summarized in the following lemma.

\begin{lem}
\label{lem:preimage cusp}
Consider $\partial V_\gamma\cap Y=\theta\cup\theta'$ where $\theta,\theta'$ are both parallel to $\gamma$ in $S$ but on different sides. Fix a basepoint $x\in\theta\subset\partial V_\gamma\cap Y$. We have the following:
\begin{enumerate}
    \item{$\pi_1(\partial V_\gamma,x)\cap\pi_1(Y,x)=\langle\theta\rangle$.}
    \item{If $\mu\in\pi_1(\partial V_\gamma,x)-\langle\theta\rangle$ and $\beta\in\pi_1(Y,x)$ is non-peripheral, then $\mu\beta\mu^{-1}\not\in\pi_1(Y,x)$.}
    \item{We have $p_Y^{-1}(\mb{C}_{\ep_Y}(\partial V_\gamma))\cap{\rm CC}(Q_Y)\subset\mb{C}_{\ep_0}(\theta)\cup\mb{C}_{\ep_0}(\theta')$.} 
\end{enumerate}
\end{lem}

\begin{proof}[Proof of Lemma \ref{lem:preimage cusp}]
{\bf Property (1)}. Consider $\zeta\in\pi_1(Y,x)$ homotopic relative to the basepoint to a loop $\zeta'\in\pi_1(\partial V_\gamma,x)$. Note that, since $\pi_1(V_\gamma,x)=\langle\theta\rangle$, the loop $\zeta'$ is homotopic within $V_\gamma$ and relative to the basepoint to $\theta^k$ for some $k$. Thus $\zeta\in\pi_1(Y,x)$ is homotopic relative to the basepoint to $\theta^k$ in $S\times[-1,1]$ and, therefore, also on $S$. As $\zeta\subset Y$, we can choose the pointed homotopy between $\zeta$ and $\theta^k$ to take place in $Y$. Thus $\zeta=\theta^k$ for some $k$.

{\bf Property (2)}. If $\mu\beta\mu^{-1}=\beta'\in\pi_1(Y,x)$, then there is a homotopy $h:(\mb{S}^1\times[0,1],\partial\mb{S}^1\times[0,1])\to (M,Y\cap M)$ such that $\mb{S}^1\times\{0\}$ maps to $\beta$, $\mb{S}^1\times\{1\}$ maps to $\beta'$, and $\{{\rm pt}\}\times[0,1]$ maps to $\mu$. Up to small perturbations, we can assume that $h$ is transverse to $\gamma_Y\times[-1,1]$. The intersection $h^{-1}(\gamma_Y\times[-1,1])$ consists of a finite number of simple closed curves on $\mb{S}^1\times[0,1]$. None of them can be homotopic to $\mb{S}^1\times\{0\}$ otherwise $\beta$ would be homotopic to one of the components of $\gamma_Y$ in $S\times[-1,1]$, hence in $S$, hence in $Y$ contradicting the assumption that $\beta$ is non-peripheral. Thus all the intersections are null-homotopic on $\mb{S}^1\times\{0\}$ and on $\gamma_Y\times[-1,1]$ (as the latter is $\pi_1$-injective) and can be removed via a homotopy of $h$ relative to the boundary. This means that the image of $h$ is in the complement of $S\times[-1,1]-\gamma_Y\times[-1,1]$ which deformation retracts to $(S-\gamma_Y)\times\{0\}$. Thus $\mu$ is homotopic relative to the endpoints into $Y$. This implies $\mu\in\pi_1(\partial V_\gamma,x)\cap\pi_1(Y,x)=\langle\gamma\rangle$ (by Property (1)) contradicting our assumption that $\mu\not\in\langle\theta\rangle$. 

{\bf Property (3)}. Let $\mc{O}\subset p_Y^{-1}(\mb{C}_{\ep_Y}(\partial V_\gamma))$ be a component that intersects $\partial{\rm CC}(Q_Y)$. As described by Thurston (see \cite[Chapter II]{CEG:notes_on_notes}), each component of the boundary of the convex core of $Q_Y$ can be parameterized as a pleated surface $f:({\rm int}(Y),\sigma)\to Q_Y$ homotopic to the inclusion ${\rm int}(Y)\times\{0\}$. Let $C\subset\partial{\rm CC}(Q_Y)$ be a component that intersects $\mc{O}$ at a point $f(z)\in\mc{O}\cap C$. The composition of $f$ with the projection map $p_Y:C\to M(\nu^-,\nu^+)$ is a $\pi_1$-injective pleated surface in $M(\nu^-,\nu^+)$. As the projection $p_Yf(z)\in p_Y(\mc{O})=\mb{C}_{\ep_Y}(\partial V_\gamma)$ lies in the $\ep_Y$-cuspidal neighborhood associated to $V_\gamma$, by Lemma \ref{lem:thick to thick} we have that $z$ lies in the $\ep_0$-thin part of $({\rm int}(Y),\sigma)$. Let $\beta\subset Y$ be the (homotopy class of the) core of the component of the $\ep_0$-thin part containing $z$. There is a loop $\beta_z$ of length $\le2\ep_0$ based at $z$ and homotopic to $\beta$. As the image $f(z)$ is contained in $\mb{C}_{\ep_Y}(\partial V_\gamma)$, the image of such loop is contained in $\mb{C}_{\ep_0}(\partial V_\gamma)$. Therefore $f(\beta)$ is homotopic to $\zeta\in\pi_1(V_\gamma)$. However, as we observed in the proof of Property (1), $\zeta$ is homotopic in $V_\gamma$ to $\gamma^k$ and, from the same argument as before, we deduce that $f(\beta)$ is homotopic to $\gamma^k$ in $Y$. By $\pi_1$-injectivity, this happens only if $\beta$ is homotopic to either $\theta$ or $\theta'$. Thus, the $\ep_0$-thin part of $z$ is actually the $\ep_0$-cuspidal part ${\rm\bf cusp}(\theta,\ep_0)\cup{\rm\bf cusp}(\theta',\ep_0)$. As the inclusion $C\subset Q_Y$ is 1-Lipschitz, we also have ${\rm\bf cusp}(\theta,\ep_0)\subset\mb{C}_{\ep_0}(\theta,Q_Y)$ and ${\rm\bf cusp}(\theta',\ep_0)\subset\mb{C}_{\ep_0}(\theta',Q_Y)$. Thus $\mc{O}$ intersects $\mb{C}_{\ep_0}(\theta,Q_Y)$ or $\mb{C}_{\ep_0}(\theta',Q_Y)$. This implies that $\mc{O}\subset\mb{C}_{\ep_0}(\theta,Q_Y)\cup\mb{C}_{\ep_0}(\theta',Q_Y)$ (as $p_Y(\partial\mb{C}_{\ep_0}(\theta)),p_Y(\partial\mb{C}_{\ep_0}(\theta'))=\mb{C}_\ep(\partial V_\gamma)$ for some $\ep\le\ep_0$ so $\partial\mc{O}$ cannot intersect $\partial\mb{C}_{\ep_0}(\theta))\cup\partial\mb{C}_{\ep_0}(\theta')$ otherwise $p_Y(\mc{O})=\mb{C}_{\ep_Y}(\partial V_\gamma)$ would intersect $\partial\mb{C}_\ep(\partial V_\gamma)$).
\end{proof}

We can now conclude the proof of the proposition. Recall that we want to prove that a slope on $\partial\mb{C}_{\ep_Y}(\partial V_\gamma)$ not homotopic into $Y$ is long by showing that it suitably lifts to an arc that intersects two boundary components of ${\rm CC}(Q_Y)$ on opposite sides of ${\rm CC}(Q_Y)$. We implement and make precise this idea as follows. 

Observe that $p_Y^{-1}(\partial\mb{C}_{\ep_Y}(\theta))\subset Q_Y$ contains exactly two components of the form $\mb{C}_{\ep}(\theta),\mb{C}_{\ep}(\theta')$ for some $\ep\in[\ep_Y,\ep_0]$. By the Filling Theorem \cite{C96}, there is a pleated surface $f:({\rm int}(Y),\sigma)\to Q_Y$ homotopic to the inclusion of ${\rm int}(Y)$ that meets $\partial\mb{C}_{\ep}(\theta)$ at a point $x=f(\overline{x})$ such that the minimal distance in $Q_0:={\rm CC}(Q_Y)-\bigcup_{\theta\subset\partial Y}{\mb{C}_{\ep_Y}(\theta)}$ of $x$ from a boundary component in $\partial{\rm CC}_0(Q_Y)$ and a boundary component in $\partial_1{\rm CC}(Q_Y)$ are equal, that is
\[
\min_{C_0\subset\partial_0{\rm CC}(Q_Y)}\{d_{Q_0}(x,C_0\cap Q_0)\}=\min_{C_1\subset\partial_1{\rm CC}(Q_Y)}\{d_{Q_0}(x,C_1\cap Q_0)\}
\]
and let $\kappa_0,\kappa_1\subset Q_0$ be two segments realizing those distances. Note that $\kappa=\kappa_0\cup\kappa_1$ satisfies the assumption of Theorem \ref{thm:effcusp} and Lemma \ref{lem:dist comp}. Thus, it has length 
\[
\ell(\kappa)\ge\frac{c_5}{|\chi(Y)|^{264}}d_Y(\delta^-,\delta^+).
\]
Hence $\ell(\kappa_0)=\ell(\kappa_1)=\ell(k)/2\ge (c_5/2|\chi(Y)|^{264})d_Y(\delta^-,\delta^+)$.

Note that (since $f$ is 1-Lipschitz and ${\rm inj}_x(Q_Y)\ge\ep$) we have ${\rm inj}(\overline{x})\ge\ep$. Let $\overline{\delta}_\theta$ be a loop based at $\overline{x}$ of length $\ell(\overline{\delta}_\theta)\le 2\log(256\pi^2|\chi(Y)|^2/\ep^2)$ not homotopic to $\theta$ (as given by Lemma \ref{lem:short gen}) and denote by $\delta_\theta:=f(\overline{\delta}_\theta)$ the image. It is a loop based at $x\in\partial\mb{C}_{\ep}(\theta,Q_Y)\cap{\rm CC}(Q_Y)$ with the same upper bound on the length. 

Consider the concatenation $\mu\star p_Y(\delta_\theta)$. Note that it is does not belong to $\pi_1(Y,p_Y(x))$ otherwise $\mu$ would be in $\pi_1(Y,p_Y(x))\cap\pi_1(\partial V_\gamma,p_Y(x))=\langle\theta\rangle$ by Lemma \ref{lem:preimage cusp} and this contradicts our assumption that $\mu$ is not homotopic to $Y$. 

Lift $\mu$ with basepoint $x$ obtaining a path $\mu'\subset\partial\mb{C}_{\ep}(\theta)$ with the other endpoint $y$. Lift $p_Y(\delta_\theta)$ to $Q_Y$ with basepoint $y$ obtaining a segment $\delta_\theta'$ with the other endpoint $z$ on a component $\mc{O}$ of $p_Y^{-1}(\mb{C}_{\ep_Y}(\partial V_\gamma))$. We claim that $z$ is not in $\mb{C}_{\ep}(\theta)$. In fact, if also $z$ lies on $\mb{C}_{\ep}(\theta)$, then we can join $y$ to $z$ in $\mb{C}_{\ep}(\theta)$ via an arc $\eta$. The concatenation $\delta_\theta'\star\eta$ is a loop in $\pi_1(Y,z)$. We deduce that $p_Y(\delta_\theta')p_Y(\eta)\in\pi_1(Y,z)$ and $p_Y(\eta)\in\pi_1(\partial V_\gamma,z)$. Hence $p_Y(\eta)\in\pi_1(\partial V_\gamma,p_Y(z))\cap\pi_1(Y,p_Y(z))=\langle\theta\rangle$ (by Lemma \ref{lem:preimage cusp}). As powers of $\theta$ lift to closed loops at every point of the pre-image of $p_Y(x)$ on $\mb{C}_{\ep}(\gamma)$ while $\eta$ is not closed by assumption we get the desired contradiction. 

We claim that $z$ is not in $\mb{C}_{\ep_0}(\theta')$. In fact, if $z\in\mb{C}_{\ep_0}(\theta')$ we can lift $\theta$ starting at $z$ obtaining a path $\theta''$ of length $\le 2\ep_0$ with endpoint $z'$. Note that $\theta'\subset\mb{C}_{3\ep_0}(\theta')$ so we can find a path $\eta\subset\mb{C}_{3\ep_0}(\theta')$ joining $z'$ to $z$ so that the concatenation $\theta'\star\eta$ is a loop homotopic to $\theta'$. The projection $p_Y(\eta)$ is a loop in $\pi_1(\partial V_\gamma,p_Y(x))$ and the projection $p_Y(\theta'\star\eta)=\theta p_Y(\eta)$ belongs to $\pi_1(Y,p_Y(x))$, thus $p_Y(\eta)\in\pi_1(Y,p_Y(x))\cap\pi_1(\partial V_\gamma,p_Y(x))=\langle\theta\rangle$ (by Lemma \ref{lem:preimage cusp}). However, we also have $p_Y(\theta'\star\eta)$ is a conjugate (in $\pi_1(Y,p_Y(x))$) of $\theta'$, thus $\theta'$ is conjugate to $\theta^k$ for some $k$. As this is not possible in $\pi_1(Y,p_Y(x))$, we get the desired contradiction.

By Lemma \ref{lem:preimage cusp}, we conclude that the component $\mc{O}$ is disjoint from the convex core ${\rm CC}(Q_Y)$. Thus $\ell(\mu\star\delta_\theta)=\ell(\mu)+\ell(\delta_\theta)$ is larger than the distance between $x$ and a component of $\partial{\rm CC}(Q_Y)\cap Q_0$ which by our choice is at least $\ell(\kappa_0)=\ell(\kappa_1)$. In conclusion, by Lemma \ref{lem:dist comp} and \ref{lem:short gen}, we get
\begin{align*}
\ell(\mu) &=\ell(\mu\star\delta_\theta)-\ell(\delta_\theta)\\
 &\ge\frac{c_5/2}{|\chi(Y)|^{264}}d_Y(\delta^-,\delta^+)-2\log\left(\frac{256\pi^2|\chi(Y)|^4}{\ep^2}\right)\\
 &\ge\frac{c_5/2}{|\chi(Y)|^{264}}d_Y(\delta^-,\delta^+)-2\log\left(\frac{256\pi^2|\chi(Y)|^4}{\ep_Y^2}\right).
\end{align*}
(In the second inequality we used $\ep\in[\ep_Y,\ep_0]$). The conclusion follows from a little bit of calculus recalling that $\ep_Y=\ep_0^{10}/2^{37}\pi^8|\chi(Y)|^{16}$. 
\end{proof}

\subsection{The hyperbolic Dehn filling step}
The last step of the proof consists of an application of the Hyperbolic Dehn Filling Theorem in the form for topologically tame hyperbolic 3-manifolds proved by Futer, Purcell, and Schleimer \cite{FPS22b}. In order to state it we need to recall the definition of {\em normalized length}.

\begin{dfn}[Normalized Length]
Let $\ep>0$ be smaller than a Margulis constant for dimension 3. Let $N$ be a compact orientable 3-manifold endowed with a complete hyperbolic metric on ${\rm int}(N)$. Consider the $\ep$-cuspidal neighborhood $\mb{C}_\ep(T)\subset{\rm int}(N)$ of a rank 2-cusp of $N$ corresponding to a torus boundary component $T\subset\partial N$. The restriction of the hyperbolic metric endowes the boundary $\partial\mb{C}_\ep(T)$ with a flat metric. For every homotopy class of simple closed curve $\mu\subset\partial\mb{C}_\ep(T)$ we define its normalized length as
\[
\frac{{\rm Length}(\overline{\mu})}{\sqrt{{\rm Area}(\partial\mb{C}_\ep(T))}}
\]
where $\overline{\mu}$ is the flat geodesic representative of $\mu$ on $\partial\mb{C}_\ep(T)$. It is a standard computation to check that the above does not depend on the particular choice of $\ep$. Additionally, it is a well-known consequence of the pared 3-manifold properties of $(N,T_1\cup\cdots\cup T_n)$, where $T_1,\cdots, T_n$ are the torus components of $\partial N$, that different simple closed curves on $\partial\mb{C}_\ep(T)$ are not homotopic in $N$ so the normalized length is well-defined for the homotopy class in $N$ of every curve homotopic into a toroidal boundary component.
\end{dfn}

We can now state the version of the Hyperbolic Dehn Filling Theorem for tame hyperbolic 3-manifolds proved in \cite{FPS22b}.

\begin{thm}[{\cite[Theorem 1.3]{FPS22b}}]
\label{thm:filling}
Fix $0<\ep\le\log(3)$ and $J>1$. Let $N$ be the interior of a compact 3-manifold and let $\Sigma\subset N$ be a link such that $N-\Sigma$ admits a hyperbolic structure. Suppose that in the hyperbolic structure on $N-\Sigma$ the {\em total normalized length} of the meridians of $\Sigma$ satisfies
\[
\frac{1}{\sum_{\gamma\subset\Sigma}{\frac{1}{L(\mu_\gamma)^2}}}\ge 4\max\left\{\frac{2\pi\cdot 6771\cosh^5(0.6\ep+0.1475)}{\ep^5}+11.7,\frac{2\pi\cdot 11.35}{\ep^{5/2}\log(J)}+11.7\right\}
\]
where $L(\mu_\gamma)$ is the normalized length of the meridian $\mu_\gamma$ of the component $\gamma\subset\Sigma$. Then $N$ admits a hyperbolic structure with the same end invariants as those of $N-\Sigma$ and for which $\Sigma$ is a geodesic link where each component has length at most $2\ep$. Moreover, there are $J$-bilipschitz inclusions
\[
\phi:(N)_{\ep}\to(N-\Sigma)_{\ep/1.2},\quad\psi:(N-\Sigma)_{\ep}\to(N)_{\ep/1.2}.
\]
\end{thm}

Note that 
\[
\frac{1}{\sum_{\gamma\subset\Sigma}{\frac{1}{L(\mu_\gamma)^2}}}\ge\frac{\min_{\gamma\subset\Sigma}\{L(\mu_\gamma)^2\}}{n}
\]
where $n$ is the number of components of $\Sigma$.

We apply Theorem \ref{thm:filling} to $N=S\times[-1,1]$ and $\Sigma=\gamma_Y\times\{0\}$ (we set $M:=N-\Sigma$) choosing $M(\nu^-,\nu^+)$ as hyperbolic structure on ${\rm int}(M)$ and constants $\ep=2\ep_Y=24\pi^2\ep_0^{10}/(2|\chi(Y)+2)^6$ and $J=2$. Let $\gamma$ be a component of $\gamma_Y\times\{0\}$. 

In order to check the condition on the normalized length we use the following.

\begin{lem}
\label{lem:area}    
Let $\ep_Y=\ep_0^{10}/2^{37}\pi^8|\chi(Y)|^{16}$. Let $\overline{\mu}$ be the flat geodesic representative of a slope $\mu\subset\partial\mb{C}_{\ep_Y}(\partial V_\gamma)$ not homotopic to $\gamma$. Then the normalized length of $\mu$ is at least
\[
\sqrt{\frac{{\rm Length}(\overline{\mu})}{\sinh(2\ep_0)}}.
\]
\end{lem}

\begin{proof}
Consider a (any) pleated surface $f:({\rm int}(Y),\sigma)\to M(\nu^-,\nu^+)$ properly homotopic to the inclusion of ${\rm int}(Y)\times\{0\}$. By properness, there exists a point $x\in{\rm int}(Y)$ such that $f(x)\in\partial\mb{C}_{\ep_Y}(\partial V_\gamma)$. By Lemma \ref{lem:thick to thick}, $x$ is contained in ${\rm\bf cusp}(\theta,\ep_0)\cup{\rm\bf cusp}(\theta',\ep_0)$ (where $\theta,\theta'\subset{\rm int}(Y)$ are the two peripheral curves homotopic to $\gamma$). Let $\gamma_x$ be a geodesic loop of length at most $2\ep_0$ based at $x$ representing $\theta$ or $\theta'$ and denote by $\gamma_{f(x)}$ the geodesic representative of $f(\gamma_x)$ with fixed basepoint. (Note that both $f(\gamma_x)$ and $\gamma_{f(x)}$ is contained in $\mb{C}_{\ep_0}(\partial V_\gamma)$). By basic hyperbolic geometry, the length of the flat geodesic representative $\overline{\gamma}$ of $\gamma_{f(x)}$ on $\partial\mb{C}_{\ep_1}(\partial V_\gamma)$ is ${\rm Length}(\overline{\gamma})=\sinh(\ell(\gamma_{f(x)}))\le\sinh(2\ep_0)$. We deduce the following elementary area bound 
\[
{\rm Area}(\partial\mb{C}_{\ep_1}(\partial V_\gamma))\le{\rm Length}(\overline{\gamma})\cdot{\rm Length}(\overline{\mu})\le\sinh(2\ep_0){\rm Length}(\overline{\mu}).
\]
Thus the normalized length of $\mu$ is at least
\[
\frac{{\rm Length}(\overline{\mu})}{\sqrt{{\rm Area}(\partial\mb{C}_{\ep_1}(\partial V_\gamma))}}\ge\sqrt{\frac{{\rm Length}(\overline{\mu})}{\sinh(2\ep_0)}}.
\]
\end{proof}

Combining Lemma \ref{lem:area} and Proposition \ref{pro:covering2}, we get the following.

\begin{cor}
\label{cor:length meridians}
The normalized lengths of the standard meridians of the components of $\gamma_Y\times\{0\}$ are at least
\[
\sqrt{\frac{c_3}{|\chi(Y)|^{264}}d_Y(\delta^-,\delta^+)-36\log(|\chi(Y)|)-c_4}.
\]
\end{cor}

Corollary \ref{cor:length meridians} enables us to check that the condition of Theorem \ref{thm:filling} is satisfied. We remark that with our choice of $\ep=2\ep_Y$ the maximum in Theorem \ref{thm:filling} is achieved by the first of the two quantities. By the remark below the theorem, the total normalized length is bounded from below by the minimal normalized length of the meridians divided by the square root of the number of components of $\gamma_Y$ which is $\le|\chi(Y)|+2\le 2|\chi(Y)|$. Thus, the condition that we want to check is
\begin{align*}
\frac{1}{2|\chi(Y)|}\left(\frac{c_3}{|\chi(Y)|^{264}}d_Y(\delta^-,\delta^+)-36\log(|\chi(Y)|)-c_4\right)\\
\ge\frac{2\pi\cdot 6771\cosh^5(1.2\ep_Y+0.1475)}{\ep_Y^5}+11.7.
\end{align*}

\begin{remark}
\label{rmk:lower bound}    
Combining a little calculus with $\ep_Y=\ep_0^{10}/2^{37}\pi^8|\chi(Y)|^{16}$, in order to make sure that the above holds it is enough to require that
\[
\frac{c_3}{2|\chi(Y)|^{265}}d_Y(\delta^-,\delta^+)-\frac{36\log(|\chi(Y)|)+c_4}{2|\chi(Y)|}\ge c_6|\chi(Y)|^{80}.
\]
where $c_6=2^{223}/\ep_0^{50}$.
\end{remark}

Hence, we can find a 2-bilipschitz diffeomorphism $Q\to M(\nu^-,\nu^+)$ defined on the complement of the thin part. Such a diffeomorphism allows us to compare the length of $\ell_Q(\gamma)$ to the length of the standard meridian on $\partial\mb{C}_{\ep_Y}(\partial V_\gamma)$. In fact, the length of the flat geodesic representative $\overline{\mu}$ of the standard meridian of $\mb{T}_{2\ep_Y}(\gamma,Q)$ is
\[
{\rm Length}(\overline{\mu})=2\pi\sinh(R)
\]
where $R$ is the normal radius of $\mb{T}_{2\ep_Y}(\gamma,Q(\nu^-,\nu^+))$. By \cite[Proposition 5.6]{FPS19}, the latter satisfies 
\[
\sinh(R)=\sqrt{\cosh(R)^2-1}\le\sqrt{\frac{\cosh(4\ep_Y)-1}{\cosh(\ell_Q(\gamma))-1}-1}\le\frac{4\ep_Y}{\ell_Q(\gamma)}.
\]
Thus ${\rm Length}(\overline{\mu})\le 8\pi\ep_Y/\ell_Q(\gamma)$.

Note that 
\[
\mb{C}_{\ep_Y}(\partial V_\gamma)\subset \phi(\partial\mb{T}_{2\ep_Y}(\gamma,Q))\subset\mb{C}_{4\ep_Y}(\partial V_\gamma)
\]
and there is a 1-Lipschitz retraction 
\[
\mb{C}_{4\ep_Y}(\partial V_\gamma)\to\partial\mb{C}_{\ep_Y}(\partial V_\gamma).
\]

As a consequence, the length of the flat geodesic representative of the meridian $\mu$ on $\partial\mb{C}_{\ep_Y}(V_\gamma)$ is at most
\[
{\rm Length}(\phi(\overline{\mu}))\le 2\cdot{\rm Length}(\overline{\mu})\le\frac{8\ep_Y}{\ell_Q(\gamma)}.    
\]
On the other hand, by Proposition \ref{pro:covering2}, the length of the flat geodesic representative of the meridian $\mu$ on $\partial\mb{C}_{\ep_1}(V_\gamma)$ is at least
\[
\frac{c_3}{|\chi(Y)|^{264}}d_Y(\delta^-,\delta^+)-36\log(|\chi(Y)|)-c_4.    
\]
Thus, we conclude that
\[
\frac{c_3}{|\chi(Y)|^{264}}d_Y(\delta^-,\delta^+)-36\log(|\chi(Y)|)-c_4\le\frac{16\pi\ep_Y}{\ell_Q(\gamma)}\le\frac{c_7}{|\chi(Y)|^{16}}\cdot\frac{1}{\ell_Q(\gamma)}
\]
where $c_7=2^{60}\ep_0^{10}$. This finishes the proof of Theorem \ref{thm:main1} modulo assembling all the constants we found in a single formula. This is what we do in the last section.

\subsection{Assembling the constants}
By Corollary \ref{cor:length meridians} and Remark \ref{rmk:lower bound}, we want to require that
\[
\frac{c_3}{2|\chi(Y)|^{265}}d_Y(\delta^-,\delta^+)-\frac{36\log(|\chi(Y)|)+c_4}{2|\chi(Y)|}\ge c_6|\chi(Y)|^{80}.
\]
This is equivalent to
\[
\frac{1}{2}d_Y(\delta^-,\delta^+)\ge|\chi(Y)|^{345}\frac{c_6}{c_3}+|\chi(Y)|^{265}\frac{36\log|\chi(Y)|+c_4}{2c_3}.
\]
If we throw away the low order terms and add the resulting constants (the right hand side of the next inequality is always bigger that the corresponding one in the previous inequality), it is enough to ask that
\[
d_Y(\delta^-,\delta^+)\ge\frac{2c_6+36+c_4}{c_3}|\chi(Y)|^{345}.
\]
By Remarks \ref{rmk:meridian} and \ref{rmk:lower bound} we can choose $c_3=\ep_0^{150}/2^{870},c_4=40\log(64/\ep_0),c_6=2^{223}/\ep_0^{50}$. Again, throwing away lower order terms $(2c_6+36+c_4)/c_3\le 4c_6/c_3=2^{1095}/\ep_0^{200}$. 

As for the length, by the formula at the end of the previous section, we get  
\begin{align*}
\ell_Q(\gamma) &\le\frac{c_7}{|\chi(Y)|^{16}}\cdot\frac{1}{\frac{c_3}{|\chi(Y)|^{264}}d_Y(\delta^-,\delta^+)-36\log(|\chi(Y)|)-c_4}\\
 &\le\frac{c_7}{|\chi(Y)|^{16}}\cdot\frac{1}{\frac{c_3}{2|\chi(Y)|^{264}}d_Y(\delta^-,\delta^+)}=\frac{2c_7}{c_3}\cdot\frac{|\chi(Y)|^{248}}{d_Y(\delta^-,\delta^+)}.
\end{align*}
Recalling that $c_7=2^{60}\ep_0^{10},c_3=\ep_0^{150}/2^{270}$ we get $2c_7/c_3=2^{331}/\ep_0^{160}$.

Lastly we relate $d_Y(\delta^-,\delta^+)$ to $d_Y(\nu^-,\nu^+)$. By Lemma \ref{lem:covering1}, we have
\[
|d_Y(\nu^-,\nu^+)-d_Y(\beta^-,\beta^+)|\le 4.
\]
By Lemma \ref{lem:shorter} and Remark \ref{rmk:shorter}, we have
\[
d_Y(\delta^\pm,\beta^\pm)\le c_2\log(|\chi(S)|)
\]
with $c_2=230\log_2(c_1)$ and $c_1=2^{109}/\ep_0^{60}$ (see Lemma \ref{lem:covering1}). Thus
\[
|d_Y(\nu^-,\nu^+)-d_Y(\delta^-,\delta^+)|\le 4+c_2\log(|\chi(S)|)
\]
which allows us to replace $d_Y(\delta^-,\delta^+)$ with $d_Y(\nu^-,\nu^+)$. This concludes the proof of Theorem \ref{thm:main1}.

\section{Proof of Theorem \ref{thm:main2}}

Lastly, using Theorem \ref{thm:main1} we quickly prove Theorem \ref{thm:main2}.

\begin{proof}[Proof of Theorem \ref{thm:main2}]
We make the following preliminary observation. As there is always a 1-Lipschitz map $(S,\sigma)\to M$ from a hyperbolic surface into $M$ homotopic to the inclusion $S\subset M$ and it is well known that the shortest essential loop on $(S,\sigma)$ has length at most $2\log(4|\chi(S)|)$, we know that for sure  $2\cdot{\rm inj}(M)\le 2\log(4|\chi(S)|)$ (twice the injectivity radius is the length of the shortest essential loop in $M$). 

We now consider two cases. Let $a,b,c$ be the constants of Theorem \ref{thm:main1}. Either $d_Y(\nu_F^-,\nu_F^+)<(a+b)|\chi(S)|^{345}$ or $d_Y(\nu_F^-,\nu_F^+)\ge (a+b)|\chi(S)|^{345}$ (which is larger than the threshold $ a|\chi(Y)|^{345}+b\log|\chi(S)|$ of Theorem \ref{thm:main1}). In the second case, by \cite[Theorem 1.2]{MT24}, we have that $d_Y(\nu_F^-,\nu_F^+)=d_W(\nu_S^-,\nu_S^+)$ for some non-annular subsurface $W\subset S$. By Theorem \ref{thm:main1}, recalling that $2\cdot{\rm inj}(M)\le\ell_M(\partial W)$, we have 
\begin{align*}
d_W(\nu_S^-,\nu_S^+) &\le\frac{c|\chi(W)|^{248}}{\ell_M(\partial W)}+b\log|\chi(S)|\\
&\le\frac{2c|\chi(W)|^{248}}{{\rm inj}(M)}+b\log|\chi(S)|=\frac{2c|\chi(W)|^{248}+b\cdot{\rm inj}(M)\log|\chi(S)|}{{\rm inj}(M)}.
\end{align*}

Using ${\rm inj}(M)\le \log(4|\chi(S)|)$ and $\log(|\chi(S)|)\log(4|\chi(S)|)\le|\chi(S)|^{248}$ we continue the chain of inequalities with 
\[
\le\frac{2c|\chi(S)|^{248}+b\cdot\log(|\chi(S)|)\log(4|\chi(S)|)}{{\rm inj}(M)}\le\frac{(2c+b)|\chi(S)|^{248}}{{\rm inj}(M)}.    
\]

Thus, the following always holds:
\begin{align*}
d_Y(\nu_F^-,\nu_F^+) &\le\max\left\{(a+b)|\chi(Y)|^{345},\frac{(2c+b)|\chi(S)|^{248}}{{\rm inj}(M)}\right\}\\
&\le (a+b)|\chi(Y)|^{345}+\frac{(2c+b)|\chi(S)|^{248}}{{\rm inj}(M)}\\
&=\frac{(a+b)|\chi(Y)|^{345}{\rm inj}(M)+(2c+b)|\chi(S)|^{248}}{{\rm inj}(M)}\le\frac{(2a+3b+2c)|\chi(S)|^{346}}{{\rm inj}(M)} &
\end{align*}
where we in the last step we used ${\rm inj}(M)\le \log(4|\chi(S)|)\le 2|\chi(S)|$. Setting $k=2a+3b+2c$ concludes the proof.
\end{proof}

\bibliographystyle{amsalpha.bst}
\bibliography{bibliography}

\Addresses

\end{document}